\documentclass[10pt,a4paper, oldfontcommands, openany]{memoir}
                                                      \usepackage[utf8]{inputenc}
                                                      \usepackage[english]{babel}
                                                      \usepackage{amsmath}
                                                      \usepackage{amsthm}
                                                      \usepackage{thmtools}
                                                      \usepackage{mathtools}
                                                      \usepackage{amsfonts}
                                                      \usepackage{amssymb}
                                                      \usepackage{makeidx}
                                                      \usepackage{graphicx}
                                                      \usepackage{tikz-cd}
                                                      \usepackage{float}
                                                      
                                                      \usepackage{hyperref}
                                                      
                                                      \usepackage{authblk}

                                                      \usepackage{xcolor}

                                                       \title{Weak Essentially Undecidable Theories of Concatenation II }
                                                        \author{Juvenal Murwanashyaka}
                                                        \affil{Department of Mathematics, University of Oslo, Norway}

                                                      \newtheorem{theorem}{Theorem} 
                                                      \newtheorem{lemma} [theorem] {Lemma}

                                                      \newtheorem{proposition} [theorem] {Proposition} 
                                                       \newtheorem{open problem} [theorem] {Open Problem}

                                                      \setsecnumdepth{subsection} 
                                                      \settocdepth{subsection}

\begin{document}

\maketitle

\begin{abstract}

We show that we can interpret concatenation theories in arithmetical theories without coding sequences.

\end{abstract}

\section{Introduction}

A computably enumerable  first-order theory is called \emph{essentially undecidable} if any consistent extension, in the same language, is undecidable (there is no algorithm for deciding whether an arbitrary sentence is a theorem). 
A computably enumerable  first-order theory is called  \emph{essentially incomplete} if any recursively axiomatizable consistent extension is incomplete. 
Since a  decidable consistent  theory can be extended to  a decidable complete  consistent theory  (see Chapter 1 of Tarski et al. \cite{tarski1953}), 
a theory is essentially undecidable if and only if it is essentially incomplete.
Two theories that are  known to be essentially undecidable are Robinson arithmetic $ \mathsf{Q} $ and the related theory $ \mathsf{R} $
(see Figure \ref{AxiomsOfRandQ}  for the axioms of $ \mathsf{R} $ and $ \mathsf{Q} $). 
The essential undecidability of $  \mathsf{R} $ and $ \mathsf{Q} $ is proved in   Chapter 2 of   \cite{tarski1953}.
In Chapter 1 of \cite{tarski1953}, Tarski introduces  interpretability as an indirect way of  showing that first-order theories are essentially undecidable. 
The method is indirect because it reduces the problem of essential undecidability of a theory $T$ to the problem of essentially undecidability of a theory $S$ which is known to be essentially undecidable. 
Interpretability between theories  is a reflexive and transitive relation and thus   induces a degree structure on the class of computably enumerable essentially undecidable  first-order theories.

\begin{figure}

\[
\begin{array}{r l  c  c r l  }
&{\large \textsf{The Axioms of } \mathsf{R} }
&
&
&
& {\large \textsf{The Axioms of } \mathsf{Q} }
\\
\\
\mathsf{ R_{1} } 
& \overline{n} + \overline{m} = \overline{n+m} 
&  
&  
& \mathsf{Q_1}  
& \forall x   y \;  [ \  x \neq y \rightarrow \mathrm{S}  x \neq \mathrm{S}  y    \ ]   
\\
\mathsf{ R_{2} } 
& \overline{n} \times \overline{m} = \overline{n \times m}  
& 
& 
& \mathsf{Q_2} 
& \forall x \;   [ \   \mathrm{S}  x \neq 0 \  ] 
\\
\mathsf{ R_{3} } 
&  \overline{n} \neq \overline{m}    \hspace{2cm}    \mbox{ if }  n \neq m
& 
& 
& \mathsf{Q_3}  
& \forall x  \;   [ \  x=0 \vee \exists y \;  [ \ x = \mathrm{S}  y \  ]     \ ]   
\\
\mathsf{ R_{4} } 
&  \forall x \;   [ \  x \leq \overline{n} \rightarrow \bigvee_{k \leq n } x = \overline{k}   \ ]  
&
& 
& \mathsf{Q_4}  
&  \forall x  \;   [ \  x+0 = x  \ ]   
\\
\mathsf{ R_{5} } 
&  \forall x \; [ \   x \leq \overline{n}  \vee \overline{n} \leq x   \   ] 
& 
& 
& \mathsf{Q_5}  
& \forall x  y \;   [ \  x+ \mathrm{S}  y = \mathrm{S}  ( x+y)   \ ] 
\\
&
&
&
& \mathsf{Q_6} 
&  \forall x  \;   [ \  x \times 0 = 0  \ ]
\\
&
&
&
& \mathsf{Q_7} 
&  \forall x  y  \;  [ \  x \times  \mathrm{S}  y =   x \times y   + x   \ ]  
\end{array}
\]

\caption{
Non-logical axioms of the first-order theories  $ \mathsf{R} $,  $ \mathsf{Q} $.
The axioms of $ \mathsf{R} $ are given by axiom schemes where 
$ n , m  , k $ are natural numbers 
and $ \overline{ n } $,  $  \overline{ m }  $,   $  \overline{ k }  $ are their canonical names.
}
\label{AxiomsOfRandQ}
\end{figure}

In   \cite{MurwanashyakaAML}, 
we introduce two theories of concatenation $ \mathsf{WD} $,  $ \mathsf{D} $ 
and show that they  are respectively  mutually interpretable  with  $ \mathsf{R} $  and  $ \mathsf{Q} $ 
(see Figure  \ref{AxiomsOfWDandD}  for the axioms of  $ \mathsf{WD} $ and $ \mathsf{D} $).
The language of    $ \mathsf{WD} $ and  $ \mathsf{D} $   is $ \lbrace 0, 1, \circ , \preceq \rbrace $ 
where $ 0$ and $1$ are constant symbols, $ \circ $ is a binary function symbol and $ \preceq $ is a binary relation symbol.
The intended model of  $ \mathsf{WD} $ and  $ \mathsf{D} $ is the free semigroup generated by two letters extended with the prefix relation. 
Extending finitely generated free semigroups with the prefix relation allows us to introduce $ \Sigma_1$-formulas which are expressive enough  to encode computations by  Turing machines 
(see Kristiansen \&  Murwanashyaka  \cite{KristiansenMurwanashyakaAML}).
$ \Sigma_1$-formulas are formulas on negation normal form where universal quantifiers occur bounded, 
i.e., they are of the form $ \forall x  \preceq t $. 
Axioms $  \mathsf{D}_4  \!  -  \!  \mathsf{D}_7 $  are essential for coding sequences in  $ \mathsf{D} $  since they allow us to work with 
$ \Sigma_0$-formulas,  formulas where all quantifiers are of the form $ \exists \preceq     t $,   $ \forall x  \preceq   t $.
In   \cite{MurwanashyakaAML}, 
we show that $ \mathsf{Q} $ is interpretable in  $ \mathsf{D} $ by  using especially   axioms   $  \mathsf{D}_4  \!  -  \!  \mathsf{D}_7 $   to restrict the universe of  $  \mathsf{D} $ to   a domain  $K$ on which the analogue of $ \mathsf{Q}_3 $ holds, 
that is, 
the  sentence  
$ \mathsf{Q}_3^{ \prime }   \equiv   \    \forall x   \;  [  \  x = 0  \;  \vee  \;  x = 1   \;   \vee  \; 
  \exists y \preceq x    \;   [    \   x = y0   \;  \vee  \;  x = y 1    \    ]     \    ]   $.
 To improve readability,  we use  juxtaposition instead of the binary function symbol  $ \circ $  of the formal language.
  Due to the existential quantifier in   $  \mathsf{Q}_3^{ \prime }   $, 
  we need to ensure that $ \Sigma_0$-formulas are    absolute for  $K$.

\begin{figure}

\[
\begin{array}{r l  c   r l  } 
&{\large \textsf{The Axioms of } \mathsf{WD} }
\\
\\
\mathsf{ WD_{1} } 
&  \overline{\alpha} \ \overline{ \beta} = \overline{ \alpha \beta} 
\\
\mathsf{ WD_{2} } 
&  \overline{\alpha} \neq \overline{\beta}    \hspace*{1cm}    \mbox{  if   }  \alpha \neq \beta
\\
\mathsf{ WD_{3} } 
&    \forall x  \;  [  \  x \preceq   \overline{\alpha}    \leftrightarrow  
 \bigvee_{   \gamma  \in \mathsf{Pref}  ( \alpha )   }  x=  \overline{ \gamma}   \   ] 
\\
\\
& {\large \textsf{The Axioms of } \mathsf{D} }
\\
\\
 \mathsf{D_1}  
& \forall x  y  z   \;   [ \  (x y) z = x (y z)    \ ]  
\\
 \mathsf{D_2}
&   \forall x y \; [ \  x \neq y \to   ( \ x   0 \neq y   0  \wedge  x   1 \neq y   1 \   )   \ ]  
\\
 \mathsf{D_3}  
& \forall x y  \;   [ \ x0 \neq y1    \  ]   
\\
 \mathsf{D_4}  
&  \forall x  \;    [ \ x  \preceq 0 \leftrightarrow x=0     \  ] 
\\
 \mathsf{D_5}  
& \forall x  \;    [ \ x  \preceq 1  \leftrightarrow x=1     \  ] 
\\
 \mathsf{D_6} 
&  \forall x y  \;   [ \ x  \preceq y 0   \leftrightarrow  (  \  x= y 0   \vee x \preceq y \  )     \  ] 
\\
 \mathsf{D_7} 
&  \forall x y  \;   [ \ x  \preceq y 1   \leftrightarrow  (  \  x= y 1   \vee x \preceq y \  )     \  ]
\end{array}
\]

\caption{
Non-logical axioms of the first-order theories  $ \mathsf{WD} $,  $ \mathsf{D} $.
The axioms of $ \mathsf{WD} $  are given by axiom schemes where 
$ \alpha , \beta  , \gamma $ are nonempty binary strings
and $ \overline{ \alpha  } $,  $  \overline{ \beta }  $, $  \overline{ \gamma }  $ are their canonical names.
$ \mathsf{Pref}  ( \alpha )  $ is the set of all nonempty prefixes of $ \alpha $. 
}
\label{AxiomsOfWDandD}
\end{figure}

Since $ \mathsf{D} $ and  $ \mathsf{Q} $  are mutually interpretable, 
we can identify   differences   between  these two theories  by investigating the interpretability degrees of  the theories 
 we obtain  by weakening  axioms  $  \mathsf{D}_4  \!  -  \!  \mathsf{D}_7 $,  $ \mathsf{Q}_3  $  
  which  are  essential  for coding  sequences in  $ \mathsf{D} $ and  $ \mathsf{Q} $. 
In addition to  $ \mathsf{D} $  and  $ \mathsf{WD} $, 
we  introduce   in   \cite{MurwanashyakaAML} two theories   $ \mathsf{ID} $, $ \mathsf{ID}^{ *}  $  
(called $ \mathsf{C} $,  $ \mathsf{BT} $, respectively,   in  \cite{MurwanashyakaAML}) 
and prove that their interpretability degrees are strictly between the degrees of  $ \mathsf{WD} $   and  $ \mathsf{D} $.
But we are not able to determine in   \cite{MurwanashyakaAML} whether   $ \mathsf{ID} $ and  $ \mathsf{ID}^* $ are mutually interpretable.
We obtain   $ \mathsf{ID} $ and $ \mathsf{ID}^{ *}  $   from $ \mathsf{D} $ by replacing    axioms 
$ \mathsf{D}_4  \! - \!   \mathsf{D}_7 $ with  respectively  the axiom schemas
\[
\mathsf{ ID }_4 
\equiv    \   
   \forall x  \;  [  \  x \preceq   \overline{\alpha}    \leftrightarrow  
 \bigvee_{ \gamma  \in \mathsf{Pref}  ( \alpha )   }  x=  \overline{ \gamma}   \   ] 
, 
\       \       
 \mathsf{ID}_4^{ * }   
\equiv     \   
 \forall x [  \  x \sqsubseteq_{ \mathsf{s} }  \overline{\alpha}  \rightarrow  
 \bigvee_{ \gamma \in   \mathsf{Sub}  ( \alpha )    }  x=  \overline{ \gamma} \   ] 
\]
where   $ \alpha $ is a nonempty binary string,  $ \overline{ \alpha } $ is canonical variable-free term that represents $ \alpha $, 
$  \mathsf{Pref}  ( \alpha )  $ denotes the set of all nonempty prefixes of $ \alpha $, 
$   \mathsf{Sub}  ( \alpha )   $ denotes the set of all nonempty substrings of $ \alpha $
  and   $ x  \sqsubseteq_{ \mathsf{s} }  y $ is shorthand for 
\[
 x = y   \;   \vee   \;   \exists   u v  \;   [     \     y = ux   \;   \vee   \;   y = xv    \;   \vee   \;   y = uxv    \      ]  
 \       .
\]
In the standard model, $  x  \sqsubseteq_{ \mathsf{s} }  y $ holds  if and only if $ x \in \mathsf{Sub} (y) $. 
It is easy   to interpret $ \mathsf{ID} $ in  $ \mathsf{ID}^* $ while it is less obvious whether  $ \mathsf{ID}^* $  is interpretable in  $ \mathsf{ID} $  since  the axiom schema  $   \mathsf{ID}_4^{ * }    $ puts strong constraints on the concatenation operator while any model of 
$ \mathsf{D}_1  \! - \!   \mathsf{D}_3 $ can always be extended to a model of  $ \mathsf{ID} $. 
In Section \ref{MutualInterpretabilityOfIDandIDStar}, 
we show  that    $ \mathsf{ID} $ and   $ \mathsf{ID}^{ *}  $   are     mutually interpretable.

Given mutually interpretability of   $ \mathsf{ID} $ and   $ \mathsf{ID}^{ *}  $,
a natural question is  whether the arithmetical  analogues  of  $ \mathsf{ID} $ and   $ \mathsf{ID}^{ *}  $  are also mutually interpretable.
We let   $ \mathsf{IQ} $ and  $ \mathsf{IQ}^* $ be the theories   we obtain from $ \mathsf{Q} $  by replacing     axiom  
$ \mathsf{Q}_3  $  with respectively  the axiom schemas
 \[
\mathsf{IQ}_3   \equiv     \  
 \forall x \;   [         \  x \leq \overline{n}   \leftrightarrow    \bigvee_{k \leq n } x = \overline{k}   \         ]  
 ,  
  \        \   
  \mathsf{IQ}_3^*
\equiv   \   
 \forall x \;   [        \  x \leq_{ \mathsf{l} }   \overline{n}   \rightarrow    \bigvee_{k \leq n } x = \overline{k}   \          ]  
\]
where  $n$ is a natural number,   $ \overline{ n }  $ is a canonical variable-free term that represents $n$, 
$  \leq $ is a fresh binary relation symbol that is realized as the less than or equal relation in the standard model and 
$ x \leq_{ \mathsf{l} }  y  \equiv   \       \exists z  \;  [   \   z + x = y    \   ]  $.
In Section \ref{MututalInterpretabilityOfIQandIQStar}, 
we show that    $ \mathsf{IQ} $ and   $ \mathsf{IQ}^{ *}  $   are   mutually interpretable.

We try to identify differences between concatenation theories and  arithmetical theories  by investigating the comparability of  $ \mathsf{ID} $ and $ \mathsf{IQ} $ with respect to interpretability. 
In Section  \ref{InterpretationOfIDInIQ}, 
we show that  $  \mathsf{IQ}  $ is expressive enough to  interpret  the theory $  \overline{  \mathsf{ID} }  $ 
we obtain by extending    $  \mathsf{ID}  $  with the axioms 
\[
\forall x y \;   [    \   x \neq y \rightarrow    (   \    0x  \neq 0y   \;   \wedge   \;    1x  \neq 1y   \    )        \      ]   
,    \       \   
\forall x y \;   [    \   0 x \neq 1  y       \      ]   
\                .
\]
Since  $  \mathsf{IQ}  $ does not have enough resources for coding general sequences, 
the interpretation we give shows  that we can think of   concatenation theories as naturally contained  in arithmetical theories. 
In Section \ref{CommutativeSemiringProperties0}, 
we show that the idea behind the interpretation of  $  \mathsf{ID}  $  in   $  \mathsf{IQ}  $ allows us to  give a very  simple interpretation of 
$  \mathsf{WD} $ in  $  \mathsf{R} $.
In Section   \ref{RecursionFreeInterpretationOfTCinQ}, 
we show that our  interpretation of    $  \mathsf{ID}  $  in   $  \mathsf{IQ}  $ extends in a natural way  to an interpretation in  
$  \mathsf{Q}  $ of  Grzegorczyk`s theory of   concatenation  $ \mathsf{TC} $     \cite{grz}
(see Figure  \ref{AxiomsOfTC} for the axioms of   $ \mathsf{TC} $). 
We can think of  $ \mathsf{D} $ as a  fragment of $ \mathsf{TC} $ since  $ \mathsf{TC} $ proves all the axioms of $ \mathsf{D} $  when we let 
$ x \preceq y   \equiv   \    x = y  \;  \vee  \;    \exists z   \;   [    \    y = x z   \    ]     $. 
The intended model of $ \mathsf{TC} $ is a  finitely generated free semigroup  with at least two generators. 
We have  not  been able to determine whether  $  \mathsf{IQ}  $ is interpretable in $  \overline{  \mathsf{ID} }  $  
and whether   $  \overline{  \mathsf{ID} }  $  is interpretable in   $  \mathsf{ID}  $.

\begin{figure}
\[
\begin{array}{r l }
&{\large \textsf{The Axioms of } \mathsf{ TC }   }
\\
\\
\mathsf{TC }_1 
& 
\forall x y z \;  [    \   x(yz) = (xy)z   \     ]  
\\
  \mathsf{TC }_2 
&
\forall x y z w  \;    [ \  
( \  xy = zw   \rightarrow    \big( \  ( \ x=z \wedge y=w) \vee  \
\\
& \hspace{2cm}
\exists u  \;   [ \  ( \ z= xu  \wedge uw = y \ ) \vee 
( \  x = zu \wedge uy = w  \ ) \  ]   \  \big)   \  ] 
\\
 \mathsf{ TC }_3
&
\forall x y  \;     [ \  xy \neq 0  \ ] 
\\
 \mathsf{ TC }_4
&
\forall x y   \;     [ \  xy \neq 1  \ ] 
\\
 \mathsf{ TC }_5
&
0 \neq 1
\end{array}
\]

\caption{
Non-logical axioms of the first-order theory   $ \mathsf{TC} $.
}
\label{AxiomsOfTC}
\end{figure}

We summarize our results in the following theorem. 
We let $ S  \leq T $  mean that $ S$ is interpretable in $T$. 
We let  $ S  < T $  mean $ S \leq T  \;   \wedge  \;   T  \not\leq S $. 
We let  $ S \cong T$  mean  $S \leq T   \;  \wedge  \;  T \leq S  $.
We let  $ \overline{ \mathsf{ID} }^{ * }    $ denote the theory we obtain from  $ \overline{ \mathsf{ID} }   $
by replacing $ \mathsf{ID}_4 $ with    $ \mathsf{ID}_4^* $.

\begin{theorem}

\[
\mathsf{R}  \cong   \mathsf{WD}        <     
 \mathsf{ID}    \cong    \mathsf{ID}^{*}  
   \leq  
   \overline{ \mathsf{ID} }    \cong    \overline{ \mathsf{ID} }^{ * }    
      \leq       
       \mathsf{IQ}      \cong     \mathsf{IQ}^{ * }  
         <   
  \mathsf{Q}  \cong   \mathsf{D}  
\         .
\]

\end{theorem}

It is not difficult to see that the two strict inequalities $ \mathsf{WD} < \mathsf{ID} $,   $ \mathsf{IQ} < \mathsf{Q} $ hold. 
If  $ \mathsf{ID}  $ were interpretable in    $ \mathsf{WD}   $,
 then  $ \mathsf{ID}_1  \!  -   \!    \mathsf{ID}_3   $ would be interpretable  in a finite subtheory of   $ \mathsf{WD} $. 
Since   any model of $  \mathsf{ID}_1  \!   -  \!    \mathsf{ID}_3     $ is infinite while any finite subtheory of   $ \mathsf{WD} $ has a finite model, 
 $ \mathsf{ID} $ is not interpretable in  $  \mathsf{WD} $. 
 Similarly, if  $ \mathsf{Q} $ were interpretable in   $ \mathsf{IQ}  $, it would be interpretable in a finite subtheory of  $ \mathsf{IQ}  $. 
 But, any finite subtheory of  $ \mathsf{IQ}  $ is interpretable in the first-order theory of the field of real numbers $ (  \mathbb{R} , 0, 1, + , \times   ) $,
 which was shown to be decidable   by Tarski   \cite{tarski1948}. 
Since  $ \mathsf{Q} $ is essentially undecidable,  it  is not interpretable in $  \mathsf{IQ} $.

\section{Preliminaries }

In this section, we clarify a number of notions that we only glossed over in the previous section.

\subsection{Notation and Terminology}

We consider the structures 
\[
 \mathfrak{D^{-} } =  (   \lbrace \boldsymbol{0}, \boldsymbol{1}  \rbrace^{ + } , \boldsymbol{0}, \boldsymbol{1}, ^\frown ) 
 \   \  \mbox{   and    }    \    \   
   \mathfrak{D} =  ( \lbrace \boldsymbol{0}, \boldsymbol{1}  \rbrace^{ + } , \boldsymbol{0}, \boldsymbol{1}, ^\frown , \preceq^{ \mathfrak{D} } ) 
   \]
  where $ \lbrace \boldsymbol{0}, \boldsymbol{1}  \rbrace^{ + } $ is the set of all finite non-empty strings over the alphabet $\lbrace  \boldsymbol{0}, \boldsymbol{1}  \rbrace $, the binary operator $^\frown $ concatenates elements of $ \lbrace \boldsymbol{0}, \boldsymbol{1}  \rbrace^{ + } $ and $ \preceq^{ \mathfrak{D} }   $ denotes the prefix relation, i.e., $x \preceq^{ \mathfrak{D} }  y $ if and only if $y = x$ or there exists $z \in  \lbrace \boldsymbol{0}, \boldsymbol{1}  \rbrace^{ + } $ such that $y = x ^\frown z $. 
The structure $ \mathfrak{D^{-} } $ is thus the free semigroup with two generators. 
We call elements of $ \lbrace \boldsymbol{0}, \boldsymbol{1}  \rbrace^{ + } $   \emph{bit strings}. 
The structures $ \mathfrak{D^{-} } $ and $ \mathfrak{D} $ are  first-order structures over the languages 
$\mathcal{L}_{ \mathsf{BT} }^{-} =   \lbrace 0, 1, \circ \rbrace $ and 
$\mathcal{L}_{ \mathsf{BT} }  =  \lbrace 0, 1, \circ, \preceq  \rbrace $, respectively.

The language of first-order arithmetic is $ \mathcal{L}_{ \mathsf{NT} } = \lbrace 0,  \mathrm{S} , +, \times \rbrace $
and we denote by $ ( \mathbb{N} , 0,  \mathrm{S} , +, \times ) $ the standard first-order structure. 
In first-order number theory, each natural number $n$ is associated with a numeral $ \overline{n} $ by recursion:
$ \overline{0} \equiv   \  0 $ and  $ \overline{n+1} \equiv   \  \mathrm{S}    \overline{n}  $. 
Each non-empty  bit string $ \alpha \in  \lbrace \boldsymbol{0}, \boldsymbol{1}  \rbrace^{ + } $ is associated by recursion with a unique $\mathcal{L}_{ \mathsf{BT} }^{-}$-term $ \overline{\alpha} $, called a \emph{biteral}, as follows: 
$ \overline{ \boldsymbol{0} }  \equiv 0$,  $ \overline{ \boldsymbol{1} }  \equiv 1$, 
$   \overline{  \alpha \boldsymbol{0} }  \equiv ( \overline{ \alpha } \circ  0 )$ and  
$   \overline{ \alpha  \boldsymbol{1} }  \equiv ( \overline{ \alpha } \circ  1 )$. 
The biterals are important if we, for example,  want to show that certain sets are definable since we then need to talk about elements of $ \lbrace \boldsymbol{0}, \boldsymbol{1}  \rbrace^{ + } $ in the formal theory.

A class is a formula  with at least  one free variable. 
Given a class $ I $ with  $n$ free variables, 
we write $ (x_1,   \ldots , x_n  )  \in I $ for $  I (x_1,  \ldots  , x_n) $. 
If $I$ has two free variables, we also write $ x I y $ for $ I(x, y ) $.
We let $ (  \exists x_1,  \ldots  , x_n )  \in I \;    \phi   $ and  $  ( \forall  x_1,  \ldots  , x_n )   \in I \;  \phi $
 be shorthand   for the formulas 
 $  \exists  x_1,  \ldots  , x_n   \;   [    \   I(x_1,  \ldots  , x_n)    \;  \wedge \;   \phi  \  ]   $ 
 and $  \forall x_1,  \ldots  , x_n  \;   [    \   I(x_1,  \ldots  , x_n) \rightarrow  \phi  \  ]   $, respectively. 
We let  $  \lbrace (x_1,  \ldots  , x_n)   \in I  :   \   \psi     \rbrace $  be shorthand for  $   I ( x_1,  \ldots  , x_n  )  \;   \wedge  \;   \psi $.

\subsection{Translations and  Interpretations}

We recall the method of relative interpretability introduced by Alfred Tarski  \cite{tarski1953}  for showing that first-order theories are essentially undecidable.
We restrict ourselves to   many-dimensional parameter-free one-piece relative interpretations. 
Let $\mathcal{L}_{1}$ and $\mathcal{L}_{2}$ be computable  first-order languages.
A \emph{relative translation}  $\tau$ from $\mathcal{L}_{1}$ to  $\mathcal{L}_{2}$ is a computable map   given by: 
\begin{enumerate}

\item An $\mathcal{L}_{2}$-formula $  \delta(x_1,  \ldots , x_m)   $ with exactly $m$ free variable. 
The formula $   \delta(x_1,  \ldots , x_m)  $ is called a domain.

\item For each $n$-ary relation symbol $R$ of $\mathcal{L}_{1}$,    an   $\mathcal{L}_{2}$-formula 
$  \psi_{R}(  \vec{x}_1   ,  \ldots ,   \vec{x}_n    )$ with exactly $mn$ free variables.
The equality symbol $=$  is treated as a binary relation symbol.

\item For each $n$-ary function symbol $f$ of $\mathcal{L}_{1}$, an $\mathcal{L}_{2}$-formula 
$\psi_{f}(  \vec{x}_1   ,  \ldots ,   \vec{x}_n ,  \vec{y}  )$ with exactly $m(n+1)$ free variables.

\item For each constant symbol $c$ of $\mathcal{L}_{1}$, an $\mathcal{L}_{2}$-formula 
$\psi_{c}(  \vec{y}  )  $ with exactly $m$ free variables.

\end{enumerate}

We extend $\tau$ to a translation of  atomic $\mathcal{L}_{1}$-formulas by mapping an   $\mathcal{L}_{1}$-term $t$ to an  
$\mathcal{L}_{2}$-formula $ ( t )^{ \tau , \vec{w} } $ with  free variables $  \vec{w}  $ that denote the value of $ t  $: 
\begin{enumerate}

\item[5.] For each  $n$-ary relation symbol $R$  of $\mathcal{L}_{1}$
\[
 \big( R(t_1,  \ldots , t_n )   \big)^{ \tau }   \equiv   \  
\exists  \vec{v}_1   \ldots  \vec{v}_n   \;     [     \    \bigwedge_{ i = 1 }^{ n }  \delta (  \vec{v}_i )    
\    \wedge    \   
 \bigwedge_{ j = 1 }^{ n } (  t_j )^{ \tau ,  \vec{v}_j    }    
\    \wedge    \    
\psi_{ R }   (  \vec{v}_1   \ldots  \vec{v}_n    )    
\         ]
\]
where  $  \vec{v}_1   \ldots  \vec{v}_n  $  are distinct variable symbols that do not occur in $ t_1,  \ldots , t_n $
and 
\begin{enumerate}

\item  for each variable symbol $x$ of  $\mathcal{L}_{1}$,   
$ \;     (x)^{ \tau ,  \vec{w} }  \equiv  \     \bigwedge_{ i = 1 }^{ m }   w_i = x_i        \; $

\item  for  each constant symbol $c$ of $\mathcal{L}_{1}$,  $    \;      (c)^{ \tau , \vec{w}  }  \equiv   \    \psi_{c} ( \vec{w}  )     \;   $

\item  for each  $n$-ary function  symbol $f$  of $\mathcal{L}_{1}$  
\begin{multline*}
 \big( f(t_1,  \ldots , t_n )   \big)^{ \tau  , \vec{w}   }   \equiv   \  
 \\
\exists   \vec{w}_1   \ldots    \vec{w}_n     \;     [     \    \bigwedge_{ i = 1 }^{ n }  \delta  (   \vec{w}_i   )    
\    \wedge    \   
 \bigwedge_{ j = 1 }^{ n } (  t_j )^{ \tau ,  \vec{w}_j  }      
\    \wedge    \    
\psi_{ f }   (  \vec{w}_1   \ldots    \vec{w}_n    ,   \vec{w}   )    
\         ]
\end{multline*}
where  $\vec{w}_1   \ldots    \vec{w}_n   $  are distinct variable symbols  that do not occur in  
$  \bigwedge_{ j = 1 }^{ n } (  t_j )^{ \tau , \vec{w}  }    \;   $.

\end{enumerate}

\end{enumerate}

We extend $\tau$ to a translation of all $\mathcal{L}_{1}$-formulas as follows: 
\begin{enumerate}

\item[6.] $( \neg \phi)^{\tau} \equiv  \ \neg \phi^{\tau} $

\item[7.] $( \phi  \oslash \psi )^{\tau} \equiv  \phi^{\tau}  \oslash  \psi^{\tau} $ for $\oslash  \in \lbrace \wedge, \vee,  \rightarrow,  \leftrightarrow \rbrace$

\item[8.] $( \exists x \;  \phi)^{\tau} \equiv  \  \exists \vec{x}  \;   [    \     \delta( \vec{x}  )   \;   \wedge  \;     \phi^{\tau}      \          ]   $

\item[9.] $( \forall x  \;   \phi)^{\tau} \equiv  \  \forall  \vec{x}   \;   [   \ \delta( \vec{x}  ) \rightarrow \phi^{\tau}  \     ]     \;   $.

\end{enumerate}
Let $ \large{S}  $  be an $\mathcal{L}_{1}$-theory and let  $ \large{T}  $ be an $\mathcal{L}_{2}$-theory. 
We say that $ \large{S} $ is \emph{(relatively)  interpretable}  in $ \large{T} $  if there exists a relative translation $\tau$ such that 
\begin{itemize}

\item[-]   $ \large{T}   \vdash \exists  \vec{x}  \;   \delta( \vec{x} )  $

\item[-] For each function symbol $f$ of $\mathcal{L}_{1}$
\[
\large{T}    \vdash \bigwedge_{i=1}^{n} \delta(  \vec{x}_i   )  \rightarrow 
\exists !   \vec{y}  \;        [      \       \delta( \vec{y} ) \wedge \psi_{f}(  \vec{x}_1 ,   \ldots  ,  \vec{x}_n  ,    \vec{y}  )         \           ]     
\           .
\]

\item[-] For each constant symbol $c$ of $\mathcal{L}_{1}$
\[
\large{T}     \vdash 
\exists ! \vec{y}    \;          [        \    \delta( \vec{y} ) \wedge \psi_{c}(  \vec{y} )             \               ]   
 \            .
\]

\item[-]  $ \large{T} $ proves $\phi^{\tau}$ for each non-logical axiom $\phi$ of $ \large{S} $. 
If equality is not translated as equality,  then $ \large{T} $ must prove the translation of each equality axiom.

\end{itemize}
If $ \large{S} $ is relatively interpretable in $ \large{T} $  and $ \large{T} $  is relatively interpretable in $ \large{S} $, 
we say that $ \large{S} $ and $ \large{T} $   are  \emph{mutually interpretable}.

The following proposition summarizes important properties of relative interpretability (see Tarski et al.  \cite{tarski1953} for the details).

\begin{proposition}
Let $ \large{S} $, $ \large{T} $ and $ \large{U} $  be computably enumerable  first-order theories. 
\begin{enumerate}
\item If  $ \large{S} $  is interpretable in $ \large{T} $  and $ \large{T} $  is consistent, then $ \large{S} $  is consistent. 

\item If $ \large{S} $ is interpretable in $ \large{T} $ and $ \large{T} $  is interpretable in $ \large{U} $, then $ \large{S} $  is interpretable in $ \large{U} $.

\item If  $ \large{S} $ is interpretable in $ \large{T} $  and $ \large{S} $ is essentially undecidable, then $ \large{T} $  is essentially undecidable. 
\end{enumerate}
\end{proposition}

\section{Mutual Interpretability of $ \mathsf{ID} $ and $ \mathsf{ID}^{*}  $}
\label{MutualInterpretabilityOfIDandIDStar}

\begin{figure}[hbt!]

\[
\begin{array}{r l  c   r l  } 
&{\large \textsf{The Axioms of } \mathsf{ID} }
\\
\\
\mathsf{ ID_{1} } 
&  \forall x  y  z   \;   [ \  (x y) z = x (y z)    \ ]  
\\
\mathsf{ ID_{2} } 
&  \forall x y \; [ \  x \neq y \to   ( \ x   0 \neq y   0  \wedge  x   1 \neq y   1 \   )   \ ]  
\\
\mathsf{ ID_{3} } 
&   \forall x y  \;   [    \     x0 \neq y1       \        ]   
\\
\mathsf{ ID_{4} } 
&    \forall x  \;  [  \  x \preceq   \overline{\alpha}    \leftrightarrow  
 \bigvee_{ \gamma  \in \mathsf{Pref}  ( \alpha )   }  x=  \overline{ \gamma}   \   ] 
\\
\\
& {\large \textsf{The Axioms of } \mathsf{ID}^* }
\\
\\
& \mathsf{ID}_1 ,   \    \mathsf{ID}_2  ,   \    \mathsf{ID}_3   
\\
 \mathsf{ID}_4^{ * }   
&   \forall x [  \  x \sqsubseteq_{ \mathsf{s} }  \overline{\alpha}  \rightarrow  
 \bigvee_{ \gamma \in   \mathsf{Sub}  ( \alpha )    }  x=  \overline{ \gamma} \   ] 
\end{array}
\] 
\caption{
Non-logical axioms of the first-order theories  $ \mathsf{ID} $ and  $ \mathsf{ID}^* $.
 $ \mathsf{ID}_4 $  and $ \mathsf{ID}_4^{ *}  $ are  axiom schemas where $ \alpha $ is a  nonempty binary string, 
 $ \mathsf{Pref}  ( \alpha )  $ is the set of all nonempty prefixes of $ \alpha $ 
 and  $  \mathsf{Sub}  ( \alpha )   $ is the set of all nonempty substrings of $ \alpha $. 
Furthermore, 
$
x  \sqsubseteq_{ \mathsf{s} } y 
 \equiv    \   
 x = y   \;   \vee   \;   \exists   u v  \;   [     \     y = ux   \;   \vee   \;   y = xv    \;   \vee   \;   y = uxv    \      ]  
$.
}
\label{AxiomsOfIDandIDStar}
\end{figure}

In this section, we show that  $ \mathsf{ID} $ and   $ \mathsf{ID}^{*}  $   are mutually interpretable
(see Figure \ref{AxiomsOfIDandIDStar} for the axioms of   $ \mathsf{ID} $ and   $ \mathsf{ID}^{*}  $). 
It is easy to see that   $ \mathsf{ID} $  is interpretable    in $ \mathsf{ID}^{*}  $.
We therefore   need to focus on the more difficult task of proving   that  
$ \mathsf{ID}^{*}  $   is interpretable  in   $ \mathsf{ID} $. 
It is more difficult to interpret  $ \mathsf{ID}^{*}  $  in   $ \mathsf{ID} $  because 
the axiom schema  $ \mathsf{ID}^{*}_4 $ puts strong  constraints on the concatenation operator  
while it is always possible to extend any model of  $ \mathsf{ID}_1$,  $ \mathsf{ID}_2$, $ \mathsf{ID}_3$  to a model of $ \mathsf{ID} $.  
For example, we can have models of  $ \mathsf{ID} $ where  there exist infinitely many pairs $ x, y $ such that 
$ x  y = \overline{ \alpha } $ for each nonempty string $ \alpha $. 
Indeed, consider the model where the universe is the Cartesian product $ \prod_{ i < \omega }  \lbrace 0, 1 \rbrace^* $, 
concatenation is componentwise and each  binary string $ \beta $ is mapped to the constant sequence 
$ ( \beta  )_{ i < \omega } $.

To interpret  $ \mathsf{ID}^{*}  $  in   $ \mathsf{ID} $, we need to use the axiom schema  $ \mathsf{ID}_4 $
in an essential way  to define a function $ \star $ that provably in   $ \mathsf{ID} $ satisfies the translation of each axiom of   
$ \mathsf{ID}^{*}  $. 
The idea is to observe that since we have   the right cancellation law in the weak form of  $ \mathsf{ID}_2 $, 
if we had an axiom schema for the suffix relation, denoted  $ \preceq_{ \mathsf{suff} } $,  analogues to   $ \mathsf{ID}_4 $, 
we could try to define $ \star $ by requiring that   $ x \star y = xy   $ only if    $  y \preceq_{ \mathsf{suff} }   xy  $. 
If $ xy  $ is a variable-free term and $ y \preceq_{ \mathsf{suff} } xy $,
then   the axiom schema for the suffix  relation  gives us a finite number of possibilities for  the value of  $ y  $. 
If we also knew  that $0$ and $ 1 $ were  atoms/ indecomposable, 
we  would  be   able to use $ \mathsf{ID}_2 $ and  $ \mathsf{ID}_3 $   to   determine  that   $x$ and  $y$ are also variable-free terms.
To make this idea work, we need to ensure that $ \star $ is associative. 
Our solution is to show that extending  $ \mathsf{ID} $ with an axiom schema for   $ \preceq_{ \mathsf{suff} }  $ and 
the axiom  $  \forall x y \;  [   \    y  \preceq_{ \mathsf{suff} }    x  y    \    ]    $
does not change the interpretability degree.

This section is organized as follows: 
In Section \ref{FirstSectionOnAtoms}, 
we show that we can extend $ \mathsf{ID} $ to a theory  $ \mathsf{ID}^{ (2) } $ with the same interpretability degree 
where $0$ and $1$ are atoms. 
In Section   \ref{SuffixRelationSection}, 
we show that  we can extend  $ \mathsf{ID}^{ (2) } $ to a theory  $ \mathsf{ID}^{ (3) } $ with the same interpretability degree 
where we have an axiom schema for the suffix  relation  $ \preceq_{ \mathsf{suff} } $
analogues to the axiom schema $ \mathsf{ID}_4 $.
In  Section \ref{SubstringRelationSection},    we extended  $ \mathsf{ID}^{ (3) } $  to a theory $ \mathsf{ID}^{ (4) } $  with the same interpretability degree  and  where we have an axiom schema for the substring relation, analogues to   $ \mathsf{ID}_4 $. 
In Section \ref{SuffixRelationIISection}, 
we use the axiom schema for the substring relation to extend   $ \mathsf{ID}^{ (4) } $  to a theory $ \mathsf{ID}^{ (5) } $  with the same interpretability degree  and where the suffix  relation  $  \preceq_{ \mathsf{suff} }  $   satisfies   additional properties. 
Finally, in Section \ref{InterpretationOfIDBarinIDSection}, 
we   show that  $ \mathsf{ID}^{ * }  $ is interpretable in  $ \mathsf{ID}^{ (5) } $.

\subsection{Atoms }
\label{FirstSectionOnAtoms}

It will prove useful later  to know that $0$ and $1$ are atoms. 
So, let $ \mathsf{ID}^{ (2) } $  be $ \mathsf{ID} $   extended with the axioms 
 \[
  \mathsf{AT0}   \equiv   \    \forall  xy    \;    [     \    xy \neq 0    \    ]   
  ,  \          \   
  \     
  \mathsf{AT1}   \equiv   \    \forall  xy    \;    [     \    xy \neq 1    \    ]  
   \          .
   \]

 \begin{lemma}  \label{AtomsLemma}

 $ \mathsf{ID} $  and    $ \mathsf{ID}^{ (2) } $ are mutually interpretable.

 \end{lemma}

\begin{proof}

Since $ \mathsf{ID}^{ (2) } $  is an  extension of  $ \mathsf{ID} $, 
it suffices to show that  $ \mathsf{ID}^{ (2) } $  is interpretable in  $ \mathsf{ID} $. 
Since  the axioms of  $ \mathsf{ID} $  are  universal sentences, 
it suffices to relativize quantification to a domain $K$ on which the sentences    $ \mathsf{AT0}  $,   $ \mathsf{AT1}  $ hold.
We obtain $K$ by successively restricting  the universe   to subclasses with nice properties.

Let 
\[
K_1 = \lbrace x  :   \   x = 0  \;   \vee \;    x= 1 \;   \vee \;    \exists y \;   [    \
  x = y0    \;   \vee \;      x = y1      \    ]       \    \rbrace 
\         .
\]
Clearly, $ 0, 1 \in K_1 $. 
Since concatenation is associative,  $K_1 $ is closed under concatenation.

Let 
\[
K_2 = \lbrace y \in K_1   :    \   
\forall x    \in K_1    \;   [    \    \bigwedge_{  a \in \lbrace 0,1  \rbrace }      x  y \neq a        \    ]  
\    \rbrace 
\         .
\]
We show that $0, 1 \in K_2 $. 
Let $ a, b  \in \lbrace 0,1 \rbrace $ and let   $x    \in K_1 $.
We need to show  $ xb \neq a $. 
Assume for the sake of a contradiction that $ xb = a $.
Since $ x \in K_1 $, 
let $ c  \in \lbrace 0,1 \rbrace  $ be such that $ x = c $ or $x = uc $ for some $ u$. 
Let $ d  \in  \lbrace 0,1 \rbrace   \setminus \lbrace c \rbrace $. 
Then, $ xb = a $ implies $ ddxb = dda $. 
By $ \mathsf{ID}_3 $ and $ \mathsf{ID}_2 $, $ ddx = dd $, 
 which contradicts  $ \mathsf{ID}_3 $. 
Thus, $ 0, 1 \in K_2 $.

We now show that $K_2 $ is closed under the maps $ x \mapsto x0 $, $ x \mapsto x1 $.
Let $ y  \in K_2 $ and let $ b \in \lbrace 0,1 \rbrace $. 
We need to show that $ yb \in  K_2 $. 
Since  $ y, b \in K_2 \subseteq K_1 $ and  $ K_1 $ is closed under concatenation, $ yb \in K_1 $. 
Now, let  $ a  \in \lbrace 0,1 \rbrace $ and let   $x    \in K_1 $.
We need to show  $ x yb \neq a $. 
Assume for the sake of a contradiction that $ x y b = a $. 
Then, $ ax yb = aa $. 
By  $ \mathsf{ID}_3 $ and $ \mathsf{ID}_2 $, 
we have $ ax y = a $, which contradicts $ y \in K_2 $ since  $ ax \in K_1 $ as 
$  a, x \in K_1 $ and $K_1 $  is closed under concatenation.
Thus, $K_2 $ is closed under the maps $ x \mapsto x0 $, $ x \mapsto x1 $.

The class $K_2 $  is not a domain since it may not be  closed under concatenation.
We obtain $K$ by restricting $K_2 $ to a subclass that contains $0$ and  $1$ and is closed under concatenation. 
Let 
\[
K  = \lbrace  w \in K_2  :   \ 
  \forall z \in K_2   \;      [    \    zw  \in K_2    \      ]  
     \     \rbrace 
     \         .
\]
We have $ 0, 1 \in K $ since $K_2 $ contains $0, 1 $ and   is closed under the maps $ x \mapsto x0 $, $ x \mapsto x1 $.
We now show that $K$ is closed under concatenation. 
Let $ w_0 , w_1 \in K$. 
We need to show that $ w_0 w_1 \in K$. 
Since $ w_0 \in K \subseteq K_2 $ and $ w_1 \in K $, we have $ w_0 w_1 \in K_2 $. 
Now, let $ z \in K_2 $.  We need to show that $ z w_0 w_1 \in K_2 $. 
We do not worry about parentheses since  $ \mathsf{ID}_1 $ tells us that concatenation is associative. 
Since $ w_0 \in K$, we have $ z w_0 \in K_2 $. 
Since $ w_1 \in K$, we have $ z w_0 w_1 \in K_2 $. 
Hence, $ w_0 w_1 \in K $. 
Thus, $ K$ is closed under concatenation. 
\end{proof}

\subsection{Suffix Relation }
\label{SuffixRelationSection}

In this section,  we   show that we can extend  $ \mathsf{ID}^{ (2) }  $ to a theory  
where we have  an axiom schema for the suffix relation, analogues to $ \mathsf{ID}_4 $, 
without changing the interpretability degree. 
We  extend the language of  $ \mathsf{ID}^{ (2) } $ with a fresh binary relation symbol  $ \preceq_{ \mathsf{suff} } $. 
Given a nonempty binary string $ \alpha $, let   $ \mathsf{Suff} ( \alpha )  $ denote the set of all nonempty suffixes of $ \alpha $: 
$ \gamma \in \mathsf{Suff} ( \alpha )   $ if and only if  $ \alpha = \gamma $ or 
$ \exists  \delta  \in \lbrace \boldsymbol{0}, \boldsymbol{1}  \rbrace^{ + }   \;   [    \ 
   \alpha = \delta \gamma   \;  \wedge  \;     \gamma      \in \lbrace \boldsymbol{0}, \boldsymbol{1}  \rbrace^{ + }       \      ]   $.
Let  $ \mathsf{ID}^{ (3)  }  $ be  $ \mathsf{ID}^{ (2)  }$ extended with the following axiom schema 
\[
 \forall x     \;        [  \  x   \preceq_{ \mathsf{suff} }       \overline{\alpha}     \leftrightarrow  
 \bigvee_{ \gamma \in   \mathsf{Suff} ( \alpha )    }  x=  \overline{ \gamma} \   ] 
 \            .
\]

\begin{lemma}   \label{SuffixLemma}

$ \mathsf{ID} $ and  $ \mathsf{ID}^{ (3) } $ are mutually interpretable.

\end{lemma}

\begin{proof}

Since $ \mathsf{ID}^{ (3) } $ is an extension of  $ \mathsf{ID} $, it suffices by  Lemma \ref{AtomsLemma} 
 to show that the suffix relation is  definable in  $ \mathsf{ID}^{ (2) } $.
We translate the suffix relation as follows: 
$ x    \preceq_{ \mathsf{suff} }       y $ if and only   if 
\begin{itemize}

\item[(1)]  $  y = x   \;  \vee  \;  \exists u    \;   [   \   y =  ux    \   ]    $

\item[(2)]   $   \forall u \preceq  x  \;      \big[   \   
u = 0  \;  \vee   \;  u = 1     \;  \vee   \;    \exists v \preceq u   \;  [    \    u = v 0   \;  \vee   \;    u = v 1   \   ]    
    \         \big]      $

\item[(3)]  $ \preceq $ is reflexive and transitive on the class  $ I_x = \lbrace z :    \   z \preceq x  \rbrace $,  $   \     x \in I_x  $
and   $ \forall z \in I_x   \;   \forall w \preceq z   \;    [    \     w \in I_x   \     ]     $.

\end{itemize}

Given a nonempty binary string $ \alpha $,   we need to show that 
\[
 \mathsf{ID}^{ (2) }    \vdash  
 \forall x   \;   [    \ 
  x   \preceq_{ \mathsf{suff} }        \overline{\alpha}     \leftrightarrow  
 \bigvee_{ \gamma \in   \mathsf{Suff} ( \alpha )    }  x=  \overline{ \gamma} 
 \      ] 
 \                 .
\]

\paragraph{$(   \Leftarrow   )$ }

We  show that 
\[
 \mathsf{ID}^{ (2) }    \vdash  
 \forall x  \;   [      \     
\big(     \        \bigvee_{ \gamma \in   \mathsf{Suff} ( \alpha )    }  x=  \overline{ \gamma}          \         \big)   
 \rightarrow   
 x   \preceq_{ \mathsf{suff} }        \overline{\alpha}    
  \      ] 
   \                 .
\]

Let   $ \gamma \in   \mathsf{Suff} ( \alpha )   $. 
We need to show that $  \overline{ \gamma }    \preceq_{ \mathsf{suff} }       \overline{ \alpha } $   holds. 
That is, we need to show that $  \overline{ \gamma }  $ and   $   \overline{ \alpha } $ satisfy  (1)-(3). 
It is easy to prove by induction on the length of binary strings  that 
\[
 \mathsf{ID} \vdash \overline{ \delta }    \,   \;    \overline{ \zeta }   =   \overline{  \delta  \zeta  }   
 \       \     
 \mbox{ for  all   }       \delta , \zeta \in    \lbrace \boldsymbol{0}, \boldsymbol{1}  \rbrace^{ + }  
   \         .
   \tag{*}
 \]
 By (*), $ \overline{ \alpha  }   =   \overline{ \gamma } $ or  $ \overline{ \alpha  }   = \overline{ \delta }  \;     \overline{ \gamma } $ 
 where $ \delta $ is a prefix of $ \alpha $. 
 Hence, (1) holds. 
By (*)  and  the axiom schema  $ \mathsf{ID}_4 $ for the prefix  relation,    $ \overline{ \gamma }  $ satisfies (2)-(3). 
Thus,  $  \overline{ \gamma }    \preceq_{ \mathsf{suff} }     \overline{ \alpha } $   holds.

\paragraph{$(   \Rightarrow  )$ }

We need to show that 
\[
 \mathsf{ID}^{ (2) }    \vdash  
 \forall x   \;   [    \ 
  x   \preceq_{ \mathsf{suff} }        \overline{\alpha}      \rightarrow  
 \bigvee_{ \gamma \in   \mathsf{Suff} ( \alpha )    }  x=  \overline{ \gamma} 
 \      ] 
  \                 .
 \tag{**}
\]

We prove (**) by induction on the length of $ \alpha $. 
Assume $ \alpha \in  \lbrace \boldsymbol{0}, \boldsymbol{1}  \rbrace $ and 
$ x   \preceq_{ \mathsf{suff} }        \overline{\alpha}     $ holds. 
By (1), $ x =   \overline{\alpha}   $ or there exist $ u  $ such that $ \overline{\alpha}   = u x $. 
By $ \mathsf{AT0} $ and $ \mathsf{AT1} $,   we have $ x =   \overline{\alpha}    $. 
Thus,   (**) holds when   $ \alpha \in    \lbrace \boldsymbol{0}, \boldsymbol{1}  \rbrace   $.

We consider the inductive case. 
Assume  $  \alpha = \beta a $ where $ a \in   \lbrace \boldsymbol{0}, \boldsymbol{1}  \rbrace $,   
$  \beta  \in  \lbrace \boldsymbol{0}, \boldsymbol{1}  \rbrace^{ + }   $    and  
\[
 \mathsf{ID}^{ (2) }    \vdash  
 \forall x   \;   [    \ 
  x   \preceq_{ \mathsf{suff} }        \overline{\beta }      \rightarrow  
 \bigvee_{ \gamma \in   \mathsf{Suff} ( \beta  )    }  x=  \overline{ \gamma} 
 \      ] 
  \                 .
 \tag{***}
\]
By definition,  $  \overline{ \alpha } =  \overline{ \beta  a }  =  \overline{ \beta  }   \,  \overline{  a  }   $.
Assume $ x  \preceq_{ \mathsf{suff} }        \overline{\alpha}   $ holds. 
By (1),  $ x =   \overline{\alpha}   $ or there exist $ u  $ such that $  \overline{\alpha}    = u x $. 
If $ x =   \overline{\alpha}  $, we are done. 
So, assume   $  \overline{\alpha}    = u x $. 
By (3), we have $ x \preceq x $. 
Then, by (2),  we have one of the following cases: 
(i)   there exists $ b \in \lbrace 0, 1 \rbrace $ such that $ b = x $, 
(ii)  there exist $ w  \preceq  x $ and $ c  \in \lbrace 0, 1 \rbrace  $   such that $ x  = w c $. 
Assume (i) holds.  
We have $  \overline{ \beta  }  \,  \overline{  a  }   =    \overline{\alpha}   = ux = u b  $. 
By  $ \mathsf{ID}_3 $, we have $ \,  \overline{  a  }   = b = x  $. 
Thus,  $  x=  \overline{ \gamma}  $  where $  \gamma \in   \mathsf{Suff} ( \alpha )     $.

Assume  (ii) holds. 
Then,  $  \overline{ \beta  }    \,  \overline{  a  }   =    \overline{\alpha}   = ux = u w c $.  
By  $ \mathsf{ID}_3 $, we have $ \overline{  a  }  = c $. 
By $ \mathsf{ID}_2 $, we have  $  \overline{ \beta  }  = u w $. 
Furthermore 
\begin{itemize}

\item[-]   $   \forall u \preceq  w   \;      \big[   \   
u = 0  \;  \vee   \;  u = 1     \;  \vee   \;    \exists v \preceq u   \;  [    \    u = v 0   \;  \vee   \;    u = v 1   \   ]    
    \         \big]      $
    since $ u \preceq w  \;   \wedge   \;   w  \preceq x $ implies   $ u \preceq x $ by (3)

\item[-]   since $ w \preceq x $ and (3) holds,     $ \preceq $ is reflexive and transitive on  the class 
$ I_w = \lbrace z :    \   z \preceq w  \rbrace $,  $   \     w \in I_w  $
and    $ \forall z \in I_w   \;   \forall w \preceq z   \;    [    \     w \in I_w   \     ]     $.

\end{itemize}
Thus,  $ w  \preceq_{ \mathsf{suff} }        \overline{\beta }     $ holds. 
By (***), $ w = \overline{ \delta } $ where $ \delta $ is a suffix of $ \beta $. 
Then, $ x = w  \overline{  a  }   =   \overline{ \delta }  \,   \overline{  a  }  =   \overline{ \delta  a}   $
 and   $\delta  a $  is a suffix of $ \alpha $. 
Thus,  $ \alpha $ satisfies (**).

Thus, by induction, (**) holds for all nonempty binary strings $ \alpha $.  
\end{proof}

\subsection{Substring Relation }
\label{SubstringRelationSection}

In this section,     we   show that we can extend  $ \mathsf{ID}^{ (3) }  $ to a theory  
where we have  an axiom schema for the substring relation, analogues to $ \mathsf{ID}_4 $, 
without changing the interpretability degree. 
We extend the language of  $ \mathsf{ID}^{ (3) } $ with a fresh binary relation symbol  $ \preceq_{ \mathsf{sub} } $. 
Given a nonempty binary string $ \alpha $, 
let   $ \mathsf{Sub} ( \alpha )  $ denote the set of all nonempty substrings  of $ \alpha $: 
$ \beta \in  \mathsf{Sub} ( \alpha )  $ if and only if 
$ \alpha  = \beta $ or there exist $ \gamma , \delta   \in \lbrace \boldsymbol{0}, \boldsymbol{1}  \rbrace^{ + }   $ 
such that 
$ \beta \in  \lbrace \boldsymbol{0}, \boldsymbol{1}  \rbrace^{ + }   $ and 
$ \alpha  = \gamma \beta     \;  \vee  \;     \alpha  =  \beta \delta     \;  \vee  \;     \alpha  = \gamma  \beta  \delta $.
Let  $ \mathsf{ID}^{ (4)  }  $ be  $ \mathsf{ID}^{ (3)  }$ extended with the following axiom schema 
\[
 \forall x    \;     [           \  x   \preceq_{ \mathsf{sub} }        \overline{\alpha}     \leftrightarrow  
 \bigvee_{ \gamma \in   \mathsf{Sub} ( \alpha )    }  x=  \overline{ \gamma}        \             ] 
 \          .
\]

\begin{lemma}   \label{SubstringLemma}

$ \mathsf{ID} $ and  $ \mathsf{ID}^{ (4) } $ are mutually interpretable.

\end{lemma}

\begin{proof}

By Lemma \ref{SuffixLemma}, 
it suffices  to show that the substring  relation is  definable in  $ \mathsf{ID}^{ (3) } $.
We translate the substring  relation as follows
\[
 x  \preceq_{ \mathsf{sub} }     y 
 \equiv   \   
 x \preceq  y     \;  \vee  \;    x \preceq_{ \mathsf{suff} }    y   \;  \vee  \; 
  \exists u \preceq y    \;     [         \    x    \preceq_{ \mathsf{suff} }     u     \             ]      
\       .
\]
By the axiom schema for the prefix relation and the axiom schema for the suffix relation, 
it is easy to see that  $  \mathsf{ID}^{ (3) } $  proves 
$ \forall x   \;   [    \ 
  x      \preceq_{ \mathsf{sub} }         \overline{\alpha}     \leftrightarrow  
 \bigvee_{ \gamma \in   \mathsf{Sub} ( \alpha )    }  x=  \overline{ \gamma} 
 \      ]  
 $ 
 for each  nonempty binary string  $ \alpha  $.
\end{proof}

\subsection{Suffix Relation II }
\label{SuffixRelationIISection}

We are finally ready to equip   the suffix relation with two very important properties. 
Let $ \mathsf{ID}^{ (5) } $ be   $ \mathsf{ID}^{ (4) } $ extended with the following axioms
\[
\forall x  \;  [    \    \bigwedge_{ a \in \lbrace 0, 1 \rbrace }     a  \preceq_{ \mathsf{suff} }   xa     \      ]  
,     \      \          \        \   
\forall x  y  \;  [    \    x    \preceq_{ \mathsf{suff} }  y   \rightarrow  
   \bigwedge_{ a \in \lbrace 0, 1 \rbrace }     xa   \preceq_{ \mathsf{suff} }   ya     \      ]  
   \                  .
\]
To show that  $ \mathsf{ID}^{ (5) } $ and $ \mathsf{ID} $ are mutually interpretable,   we need the following lemma. 
Recall that   a class  is  a formula with at least one free variable  and that  if  $ I $ is a class with one free variable  we occasionally  write 
$ x \in I $ for $ I(x) $.

\begin{lemma}  \label{TheClassJ}

There exists  a class $J$ with the following properties: 
\begin{itemize}

\item[\textup{(1)}] $  \;   \mathsf{ID}^{ (4) }   \vdash     t \in J        \       $   for each variable-free term $t$

\item[\textup{(2)}]   $  \mathsf{ID}^{ (4) }   \vdash   \forall x  \;  \forall z  \in J   \;    \big [        \  
      \bigwedge_{ a \in \lbrace 0,  1 \rbrace }   
 (       \    z = xa  \rightarrow    a   \preceq_{ \mathsf{suff} }   z    \    )       \     \big]    $

\item[\textup{(3)}]   $  \mathsf{ID}^{ (4) }   \vdash   \forall x y  \;   \forall z  \in J   \;      \big [        \    
    \bigwedge_{ a \in \lbrace 0,  1 \rbrace }   
    \big(     \    
(     \      z = ya \;   \wedge   \;     x    \preceq_{ \mathsf{suff} }   y     \     )  
      \rightarrow   
(      \            xa   \preceq_{ \mathsf{suff} }   z      \     )  
\       \big) 
   \big]    $

\item[\textup{(4)}]   $  \mathsf{ID}^{ (4) }   \vdash    \forall z  \in J   \;      \big [        \    
z = 0  \;   \vee  \;  z = 1  \;   \vee  \;    \exists u   \preceq_{ \mathsf{sub} }  z    \;   [        \   
z = u0      \;   \vee  \;    z = u 1     \          ]     
 \          \big]    $

\item[\textup{(5)}]   $  \mathsf{ID}^{ (4) }   \vdash    \forall z  \in J   \;     \forall u \;      \big [        \    
u   \preceq_{ \mathsf{sub} }  z   \rightarrow   u  \in  J     
 \          \big]    $.
\end{itemize}

\end{lemma}

\begin{proof}

We define $J$ as follows: 
 $   u \in J  $ if and only if 
\begin{itemize}

\item[(i)]  $    u   \preceq_{ \mathsf{sub} }  u    $

      \vspace*{0.06cm}

\item[(ii)]  $  \forall w   \preceq_{ \mathsf{sub} }   u  \;    [   \   w    \preceq_{ \mathsf{sub} }     w    \   ]   $

      \vspace*{0.06cm}

\item[(iii)]    $    \forall w    \preceq_{ \mathsf{sub} }      u  \;    \forall v_0   \preceq_{ \mathsf{sub} }  w  \; 
  \forall v_1    \preceq_{ \mathsf{sub} }   v_0   \;  [   \    
 v_1    \preceq_{ \mathsf{sub} }    w    \   ]       \;   $

       \vspace*{0.06cm}

 \item[(A)]    $     \forall w    \preceq_{ \mathsf{sub} }    u   \;  [      \    
 w = 0  \;  \vee   \;   w = 1  \;  \vee   \;  
  \exists  v     \preceq_{ \mathsf{sub} }      w  \;     [    \    w = v0 \;  \vee   \;   w = v1  \    ]     
     \          ]       \;  $.

 \item[(B)]   $   \forall w    \preceq_{ \mathsf{sub} }      u  \;        \forall x    \;       [    \   
      w = x 0          \rightarrow     0   \preceq_{ \mathsf{suff} }   w         \      ]     
         $

      \vspace*{0.06cm}

 \item[(C)]   $  \forall w    \preceq_{ \mathsf{sub} }      u  \;        \forall x    \;       [    \   
      w = x 1          \rightarrow     1   \preceq_{ \mathsf{suff} }   w         \      ]     
         $

      \vspace*{0.06cm}

\item[(D)]   $   \forall w    \preceq_{ \mathsf{sub} }      u  \;        \forall x   y   \;      [    \   
 (       \         w = y 0      \;   \wedge   \;   x   \preceq_{ \mathsf{suff} }    y    \          )   
     \rightarrow      
      x0   \preceq_{ \mathsf{sub} }     w   
         \      ]   
         $

             \vspace*{0.06cm}

\item[(E)]   $    \forall w    \preceq_{ \mathsf{sub} }      u  \;        \forall x   y   \;      [    \   
 (       \         w = y 1      \;   \wedge   \;   x   \preceq_{ \mathsf{suff} }    y    \          )   
     \rightarrow      
      x1   \preceq_{ \mathsf{sub} }     w   
         \      ]   
        \;    $.

\end{itemize}

It follows straight from the definition that $J$ satisfies clauses (2)-(4). 
By the axiom schema for the substring relation,  the axiom schema for the suffix relation,  
$  \mathsf{AT0} $,  $  \mathsf{AT1} $,  $  \mathsf{ID}_2  $  and   $  \mathsf{ID}_3 $, 
$J$ satisfies Clause (1). 
It remains to show that $J$ also satisfies Clause (5). 
That is, we  need to show that $J$ is downward closed under    $  \preceq_{ \mathsf{sub} }   $. 
So, assume $ u^{ \prime }  \preceq_{ \mathsf{sub} }   u \in J $. 
We need to show that   $ u^{ \prime }   \in J $. 
That is, we  need to show that $ u^{ \prime }  $ satisfies (i)-(iii) and (A)-(E). 
We show that  $ u^{ \prime } $ satisfies (i). 
Since $u $ satisfies (ii),  $  u^{ \prime }   \preceq_{ \mathsf{sub} }   u   $ implies $ u^{ \prime }   \preceq_{ \mathsf{sub} }   u^{ \prime } $. 
Thus, $ u^{ \prime } $ satisfies (i).

We show that  $ u^{ \prime } $ satisfies (ii)-(iii) and (A)-(E).
Consider one of these clauses. 
It is of the form   $  \forall w  \preceq_{ \mathsf{sub} }      u^{ \prime }   \;   \phi (w)   $. 
We need to show that   $  \forall w  \preceq_{ \mathsf{sub} }      u^{ \prime }   \;   \phi (w)   $ holds. 
Since  $u \in J $, we know that  $ \forall  w  \preceq_{ \mathsf{sub} }      u   \;   \phi (w)  $ holds. 
Let $ w  \preceq_{ \mathsf{sub} }      u^{ \prime }   $. 
We need to show that  $ \phi (w )  $ holds.
Since  $ \forall  w  \preceq_{ \mathsf{sub} }      u   \;   \phi (w)  $ holds, 
 it suffices to show that  $ w  \preceq_{ \mathsf{sub} }   u $ holds. 
By assumption 
\[
 w  \preceq_{ \mathsf{sub} }    u^{ \prime }   \preceq_{ \mathsf{sub} }    u    
 \          .
 \]
 Since $u$ satisfies (i)
\[
 w  \preceq_{ \mathsf{sub} }    u^{ \prime }   \preceq_{ \mathsf{sub} }    u    \preceq_{ \mathsf{sub} }   u 
 \          .
 \]
 Then,   $ w     \preceq_{ \mathsf{sub} }     u $  since $u$ satisfies (iii). 
Hence,  $  \forall w  \preceq_{ \mathsf{sub} }      u^{ \prime }   \;   \phi (w)   $   holds. 
Thus,  $ u^{ \prime } $ satisfies clauses (ii)-(iii), (A)-(E).

Since $ u^{ \prime } $ satisfies (i)-(iii) and (A)-(E),    $ u^{ \prime } \in   J   $. 
Thus,    $J $ is downward closed under $  \preceq_{ \mathsf{sub} }     $. 
\end{proof}

\begin{lemma} \label{SecondSuffixLemma}

 $ \mathsf{ID} $ and   $ \mathsf{ID}^{ (5) } $   are mutually  interpretable.

\end{lemma}

\begin{proof}

By Lemma \ref{SubstringLemma}, 
it suffices to show that  $ \mathsf{ID}^{ (5) } $   is interpretable in $ \mathsf{ID}^{ ( 4 )  }  $. 
Let $J$ be the class given by Lemma  \ref{TheClassJ}.
To  interpret  $ \mathsf{ID}^{ (5) } $   in $ \mathsf{ID}^{ ( 4 )  }  $  it suffices to translate the suffix relation  as follows
\[
x \preceq_{ \mathsf{suff} }^{ \tau }  y   
\equiv   \   
(    \   y \in J   \;   \wedge   \;   x   \preceq_{ \mathsf{suff} }  y     \     ) 
\     \vee    \   
(    \  y \not\in J     \;   \wedge   \;    x =  x     \        ) 
\          .
\]

We need show that the translation of each instance of  the axiom schema for the suffix relation is a theorem of  $ \mathsf{ID}^{ ( 4 )  }  $.
Let $ \alpha $ be a nonempty binary string. 
We need to show that 
\[
 \forall x     \;        [  \  x   \preceq_{ \mathsf{suff} }^{ \tau }        \overline{\alpha}     \leftrightarrow  
 \bigvee_{ \gamma \in   \mathsf{Suff} ( \alpha )    }  x=  \overline{ \gamma} \   ] 
 \tag{A} 
\]
holds. 
By  Clause (1) of Lemma   \ref{TheClassJ},  $   \;     \overline{ \alpha } \in J $. 
Hence, by the definition of  $ \preceq_{ \mathsf{suff} }^{ \tau }   $,   (A)  holds if and only if 
\[
 \forall x     \;        [  \  x   \preceq_{ \mathsf{suff} }       \overline{\alpha}     \leftrightarrow  
 \bigvee_{ \gamma \in   \mathsf{Suff} ( \alpha )    }  x=  \overline{ \gamma} \   ] 
 \tag{B} 
\]
holds.
Observe that (B) is an instance of the axiom schema for the suffix relation. 
Thus, 
the translation of each instance of  the axiom schema for the suffix relation is a theorem of  $ \mathsf{ID}^{ ( 4 )  }  $.

We need to show that the translation of the axiom 
\[
\forall x  \;  [    \    \bigwedge_{ a \in \lbrace 0, 1 \rbrace }     a  \preceq_{ \mathsf{suff} }   xa     \      ]  
\tag{C}
\]
is a theorem of  $ \mathsf{ID}^{ ( 4 )  }  $. 
Let $x $ be arbitrary and let  $  a \in \lbrace 0, 1 \rbrace $. 
We need to show that $ a  \preceq_{ \mathsf{suff} }^{ \tau }  xa $ holds. 
Assume  $ xa \in J $.
 Then,    $ a  \preceq_{ \mathsf{suff} }^{ \tau }  xa $  holds if and only if $ a    \preceq_{ \mathsf{suff} }  xa $ holds.
By Clause (2)  of   Lemma   \ref{TheClassJ},    $ a    \preceq_{ \mathsf{suff} }    xa $ holds. 
Hence, $ a  \preceq_{ \mathsf{suff} }^{ \tau }  xa $ holds when  $ xa \in J $.
Assume now  $ xa   \not\in J $.
Then, $ a  \preceq_{ \mathsf{suff} }^{ \tau }  xa $  holds  by the second disjunct in the definition of  $  \preceq_{ \mathsf{suff} }^{ \tau }  $. 
Thus, the translation of  (C) is a theorem of     $\mathsf{ID}^{ ( 4 )  }  $.

We need to show that the translation of the axiom 
\[ 
\forall x  y  \;  [    \    x    \preceq_{ \mathsf{suff} }  y   \rightarrow  
   \bigwedge_{ a \in \lbrace 0, 1 \rbrace }     xa   \preceq_{ \mathsf{suff} }   ya     \      ]  
   \tag{D}
\]
is a theorem of  $ \mathsf{ID}^{ ( 4 )  }  $. 
Let   $ a \in \lbrace 0, 1 \rbrace  $  and  assume  $ x   \preceq_{ \mathsf{suff} }^{ \tau }   y $. 
We need to show that    $  xa   \preceq_{ \mathsf{suff} }^{ \tau }   ya    $  holds.
Assume first $ ya  \not\in J $. 
Then,   $  xa   \preceq_{ \mathsf{suff} }^{ \tau }   ya    $ holds  by the second disjunct in the definition of  
$  \preceq_{ \mathsf{suff} }^{ \tau }  $. 
Assume next $ ya \in J $. 
Then, by  Clause  (4)   of   Lemma   \ref{TheClassJ},   
 $ ya  \in \lbrace  0, 1 \rbrace $   or 
   there exist $ u \preceq_{ \mathsf{sub} } ya $  and $ b   \in \lbrace  0, 1 \rbrace $   such that  $ ya = u b $. 
By $ \mathsf{AT0} $,  $ \mathsf{AT1} $ and   $ \mathsf{ID}_3 $,    
we have  $ ya = u a $ where  $ u \preceq_{ \mathsf{sub} } ya $.
By  $ \mathsf{ID}_2 $,  we have $ y = u $. 
Hence, $  y  \preceq_{ \mathsf{sub} } ya $. 
By Clause  (5)   of   Lemma   \ref{TheClassJ},     $ y \in J $. 
Thus, since $ x   \preceq_{ \mathsf{suff} }^{ \tau }   y $ holds and $ y \in J $, 
we have  $ x   \preceq_{ \mathsf{suff} }  y $ by the definition of $  \preceq_{ \mathsf{suff} }^{ \tau } $. 
Then,  by  Clause  (3)   of   Lemma   \ref{TheClassJ},  $  xa   \preceq_{ \mathsf{suff} }   ya    $  holds.
Thus, the translation of  (D) is a theorem of     $\mathsf{ID}^{ ( 4 )  }  $. 
\end{proof}

\subsection{Interpretation of $ \mathsf{ID}^{ * }  $ in $ \mathsf{ID} $ }
\label{InterpretationOfIDBarinIDSection}

We are finally ready to show that  $ \mathsf{ID}^{ * }  $ and  $ \mathsf{ID} $ are mutually interpretable.

\begin{theorem}

The theories $ \mathsf{ID} $, $ \mathsf{ID}^{ * }  $ are mutually interpretable.

\end{theorem}

\begin{proof}

To interpret  $ \mathsf{ID} $   in  $ \mathsf{ID}^{ * } $, 
it suffices to translate $ \preceq $ as follows 
\[
x \preceq y   \equiv    \    y =x   \;   \vee   \;    \exists z  \;   [     \    y = xz    \   ]  
\         .
\]
Given a nonempty binary string $ \alpha $, 
we have 
\[
\begin{array}{l c l c l }
x  \preceq    \overline{ \alpha }  
& \Leftrightarrow
& \overline{ \alpha } = x    \vee \exists z   \;     [ \     \overline{ \alpha } = x z    \  ]  
\\
\\
& \Leftrightarrow
& \overline{ \alpha } = x     \;     \vee    \;   
\\
 && 
 \exists z  \;      [      \ 
 x \sqsubseteq_{ \mathsf{s} } \overline{ \alpha } 
 \;  \wedge    \;   
 z  \sqsubseteq_{ \mathsf{s} } \overline{ \alpha }  
 \;    \wedge     \;   
    \overline{ \alpha } = x z    \         ]  
  & 
  & ( \mbox{def. of  }   \sqsubseteq_{ \mathsf{s} }    \;   )
  \\
  \\
& \Leftrightarrow 
&  \overline{ \alpha } = x       \;     \vee    \;  
\\
 &&  \bigvee_{ \beta, \gamma  \in   \mathsf{Sub} ( \alpha )    }   
 (    \   x= \overline{ \beta } \wedge  z = \overline{ \gamma } \wedge  
  \overline{ \beta }  \   \overline{ \gamma } = \overline{ \alpha }   \    ) 
& 
&  (  \mathsf{ID}^{ * }_4   )
  \\
\\
& \Leftrightarrow
&  \bigvee_{ \beta   \in \mathsf{Pref} ( \alpha )   }  x = \overline{ \beta   } 
  & 
  & (  \mathsf{ID}_1    \! -   \!    \mathsf{ID}_3  )
\          .
\end{array}   
\]
This shows  that the translation of each instance of the axiom schema $ \mathsf{ID}_4 $ is a theorem of  $  \mathsf{ID}^{ * }   $.
Thus, $ \mathsf{ID} $ is interpretable in $ \mathsf{ID}^{ * }  $.

Next,   we  show  that   $ \mathsf{ID}^{ * }   $ is interpretable in $ \mathsf{ID} $. 
By Lemma \ref{SecondSuffixLemma}, 
it suffices to show that  $ \mathsf{ID}^{ * }    $ is interpretable in $ \mathsf{ID}^{ (5) } $.
Since the axioms of $ \mathsf{ID}^{ * }  $ are universal sentences or sentences where existential quantifiers occur in the antecedent
(instances of  $ \mathsf{ID}^{ * }_4$), 
to interpret   $\mathsf{ID}^{ * }  $ in  $ \mathsf{ID}^{ (5) } $  it suffices to relativize quantification to a suitable  domain   $K$.

We start by defining an auxiliary class $K_1 $
(this is why we extended  $ \mathsf{ID}^{ (4) } $ to $ \mathsf{ID}^{ (5) } $).
Let 
\[
K_1  = \lbrace u :    \      \forall x      \;   [      \     u  \preceq_{ \mathsf{suff}  }    xu      \      ]       \    \rbrace 
\         . 
\]
By the axiom  
$ \forall x  \;  [    \    \bigwedge_{ a \in \lbrace 0, 1 \rbrace }     a  \preceq_{ \mathsf{suff} }   xa     \      ]   $, 
we have $ 0, 1 \in J $. 
We show that   $K_1 $ is closed under the maps $ u \mapsto u0 $,   $  \;   u \mapsto u1  $.
Let $ b  \in \lbrace 0, 1 \rbrace $ and let $ u \in K_1 $. 
We need to show that $ ub  \in K_1 $. 
That is, we need to show that   $ ub  \preceq_{ \mathsf{suff}  }    xu b  $ for all $ x $. 
Since $ u \in K_1 $, we know that 
\[
\forall x      \;   [      \     u  \preceq_{ \mathsf{suff}  }    xu      \      ]    
\tag{*} 
\]
holds. 
Then, by (*) and the axiom 
\[
\forall x  y  \;  [    \    x    \preceq_{ \mathsf{suff} }  y   \rightarrow  
   \bigwedge_{ a \in \lbrace 0, 1 \rbrace }     xa   \preceq_{ \mathsf{suff} }   ya     \      ]  
\]
we have 
\[
\forall x      \;   [      \     ub   \preceq_{ \mathsf{suff}  }    xub       \      ]    
\       .
\]
Hence, $ ub  \in K_1 $. 
Thus,  $K_1 $ is closed under the maps $ u \mapsto u0 $,   $  \;   u \mapsto u1  $.

The class $K_1 $ is not a domain since it may not be  closed under concatenation. 
We let 
\[
K = \lbrace u \in K_1  :    \     \forall v \in K_1   \;   [    \      vu   \in K_1   \    ]        \      \rbrace
\           .
\]
Since $ K_1 $ contains $0$ and $1$ and is closed under the maps  $ x \mapsto x0 $,   $  \;   x \mapsto x1  $, 
we have  $ 0,1 \in K $. 
We show that $K$ is closed under concatenation. 
Let $ u_0, u_1 \in K$. 
We need to show that $ u_0 u_1 \in K$. 
We start by showing that $ u_0 u_1 \in K_1 $. 
We have $ u_0 \in K \subseteq K_1$.
Hence,  $ u_0 u_1 \in K_1 $ since $ u_1 \in K $. 
Next, we need to show that $ \forall v \in K_1 [   \    v u_0 u_1 \in K_1   \    ]    $.
We do not need to worry about parentheses since $ \mathsf{ID}_1 $ tells us that concatenation is associative. 
Let $ v \in K_1 $. 
We need to show that $  v u_0 u_1  \in K_1 $. 
Since $ u_0 \in K$, we have $ v u_0 \in K_1 $.
Since $ u_1 \in K$, we have $ v u_0  u_1  \in K_1 $.
Hence, $ u_0 u_1 \in K$. 
Thus, $ K$ is closed under concatenation and therefore satisfies the domain conditions.

Since the axioms $ \mathsf{ID}_1  $, $ \mathsf{ID}_2 $, $ \mathsf{ID}_3 $ are universal sentences, 
their restrictions to $K$ are theorems of  $ \mathsf{ID}^{ (5) } $. 
It remains to show  that  the restriction to $K$  of each instance of 
\[
\mathsf{ID}^{ * }_4   \equiv    \  
\forall x [  \  x \sqsubseteq_{ \mathsf{s} }  \overline{\alpha}  \rightarrow  
 \bigvee_{ \gamma \in   \mathsf{Sub}  ( \alpha )    }  x=  \overline{ \gamma} \   ] 
\]
is a theorem of $ \mathsf{ID}^{ (5) } $.
It suffices to show that  for each  nonempty binary string   $   \alpha  $
\[
\forall x  , y \in K    \big[    \    
  xy =   \overline{ \alpha    }    \     \rightarrow    \ 
\bigvee_{ \beta , \gamma  \in     \mathsf{Sub}  ( \alpha )   }      
(    \    x = \overline{ \gamma  }  \;   \wedge   \;   y =  \overline{  \beta  }  \    ) 
  \      \big]   
  \       .
  \tag{**} 
\]
So,  let $ x , y \in K$ and  assume $ x y =  \overline{ \alpha    }      $. 
Since $ y \in K \subseteq K_ 1 $,    
we know that  $ y   \preceq_{ \mathsf{suff} } xy =    \overline{ \alpha    }   $.
By the axiom schema for the suffix relation, $ y = \overline{\beta} $ where $ \beta $ is a nonempty suffix of $ \alpha $.
So, $ x  \overline{\beta}   = \overline{\alpha}    \;  $.
By  $ \mathsf{ID}_1 $,  $ \mathsf{ID}_2 $, $ \mathsf{ID}_3 $, $ \mathsf{AT0} $, $ \mathsf{AT1} $, 
we have that $ x  = \overline{ \gamma } $ where $ \gamma $ is a nonempty prefix of $ \alpha $ such that $ \alpha = \gamma \beta $.
Thus, (**) hols for all nonempty binary strings $ \alpha $.
Thus, the translation of each instance of  $ \mathsf{ID}^{ * }_4    $ is a theorem of   $ \mathsf{ID}^{ (5) } $.
\end{proof}

\subsection{The Theories  $  \overline{  \mathsf{ID}  } $,   $  \overline{  \mathsf{ID}  }^{ * }  $  }

The axioms  $  \mathsf{ID}_1 $, $  \mathsf{ID}_2 $, $  \mathsf{ID}_3 $ describe a right cancellative semigroup. 
It is also natural to consider semigroups that are also left cancellative, 
for example 
$  (   \lbrace \boldsymbol{0}, \boldsymbol{1}  \rbrace^{ + } , \boldsymbol{0}, \boldsymbol{1}, ^\frown )  $.
Let  $\overline{ \mathsf{ID}  }$ and  $\overline{ \mathsf{ID}   }^{*}   $  be  $ \mathsf{ID} $ and $ \mathsf{ID}^{ * }  $, 
respectively,   extended with the axioms 
\[
\forall x y \;   [    \   x \neq y \rightarrow    (   \    0x  \neq 0y   \;   \wedge   \;    1x  \neq 1y   \    )        \      ]   
,    \       \   
\forall x y \;   [    \   0 x \neq 1  y       \      ]   
\                .
\]
It is not difficult to see that $ \mathsf{TC} $ proves each axiom of  $\overline{ \mathsf{ID}   }^{*}   $.
It is easily seen that our interpretation of   $  \mathsf{ID}^{*}   $ in  $ \mathsf{ID}   $ 
is also an interpretation of $\overline{ \mathsf{ID}   }^{*}   $ in  $\overline{ \mathsf{ID}  }  $.
Thus,  $\overline{ \mathsf{ID}  }  $ and  $\overline{ \mathsf{ID}   }^{*}   $ are mutually interpretable. 
We have not been able to determine whether   $\overline{ \mathsf{ID}  }  $   is interpretable  in   $   \mathsf{ID}   $.

\begin{theorem}

$\overline{ \mathsf{ID}  }   $  and   $\overline{ \mathsf{ID}   }^{*}   $ are mutually interpretable. 

\end{theorem}

\begin{open problem}

Is  $\overline{ \mathsf{ID}  }  $    interpretable  in   $   \mathsf{ID}   $?

\end{open problem}

\section{Mutual Interpretability   of $ \mathsf{IQ} $  and  $ \mathsf{IQ}^* $}
\label{MututalInterpretabilityOfIQandIQStar}

In this section, we show that  $ \mathsf{IQ}   $ and  $ \mathsf{IQ}^{ * }   $   are also  mutually interpretable.  
Recall that   $ \mathsf{IQ}^{ * }   $ is  the theory we obtain from  $ \mathsf{IQ} $ 
by removing $ \leq $ from the language and   replacing the axiom schema   $ \mathsf{IQ}_3  $ with the axiom schema 
\[
\mathsf{IQ}^{ * }_3   \equiv    \    
\forall x \;   \big[    \    x \leq_{  \mathsf{l}  }    \overline{n}  \rightarrow      \bigvee_{ k \leq n }  x = \overline{k}    \      \big]
\]
where   $ x \leq_{  \mathsf{l} }   y  \equiv     \           \exists  z \;   [      \   z + x = y   \       ]    $.

\begin{theorem}

$  \mathsf{IQ} $ and   $  \mathsf{IQ}^* $ are  mutually interpretable. 

\end{theorem}

The proof strategy is similar to  the one we used to interpret  $ \mathsf{ID}^{ * }   $  in $ \mathsf{ID}   $.
Since we obtain an interpretation of  $  \mathsf{IQ} $ in    $  \mathsf{IQ}^* $ by translating $ \leq $ as  $  \leq_{ \mathsf{l} }  $, 
we just need to focus on proving  that   $  \mathsf{IQ}^* $ is interpretable in  $  \mathsf{IQ} $.
The proof is structured as follows: 
In Section \ref{ClosureUnderAdditionAndMultiplicationLemmaSection}, 
we extend $ \mathsf{IQ} $ to a theory $  \mathsf{IQ}^+ $ which  proves that for  each inductive class there exists a an inductive subclass that is closed under addition and multiplication. 
A class  is inductive if  it contains $0$ and is closed under  the successor function. 
In Section \ref{LengthRelationLemmaSection}, 
we extend   $ \mathsf{IQ}^+ $ to  a theory  $ \mathsf{IQ}^{++} $ with the same interpretability degree as    $ \mathsf{IQ}^+ $ 
and where the ordering  relation $ \leq $ satisfies additional properties. 
In Section \ref{InterpretationOfIQStarinIQPlussSection}, 
we show that  $  \mathsf{IQ}^* $ is interpretable in  $ \mathsf{IQ}^+ $. 
Finally, in Section  \ref{CommutativeSemiringPropertiesII}, 
we show that  $ \mathsf{IQ} $ is mutually interpretable with a theory  $  \mathsf{IQ}^{ (2) }  $ that  is an extension of  $ \mathsf{IQ}^+ $.

\subsection{Closure under Addition and Multiplication}
\label{ClosureUnderAdditionAndMultiplicationLemmaSection}

A class $X$ is called inductive if $ 0 \in X $ and $ \forall x \in X   \;   [    \   \mathrm{S}x  \in X   \  ] $.
A class $X$ is called a cut if it is inductive and 
$  \forall x  \in  X    \;  \forall y   \;   [    \     y  \leq_{  \mathsf{l}  }   x  \rightarrow    y \in X     \    ]    $.
Let  $  \mathsf{IQ}^{ + } $ and  $  \mathsf{Q}^{ + } $ be respectively   $  \mathsf{IQ} $ and  $  \mathsf{Q}$ 
extended with the following axioms
\begin{itemize}

\item[-]   Associativity  of  addition   $ \forall x y z \;   [      \    (x+y) + z = x + (y+z)      \         ]  $

\item[-]   Left distributive law   $  \forall x y z \;   [      \  x   (y+z)  = xy + xz   \         ]  $

\item[-]   Associativity  of  multiplication   $  \forall x y z \;   [      \    (xy) z = x  (yz)      \         ]       \;     $.

\end{itemize}
Lemma V.5.10  of   Hajek  \& Pudlak   \cite{HajekandPudlak2017}) says that $ \mathsf{Q}^+ $ proves that  any  inductive class has a subclass that is a cut and is closed under $ +   $  and  $  \times   $. 
The   proof of that lemma shows   that $ \mathsf{IQ}^+ $ proves that  any  inductive class has an inductive  subclass that is closed under $
  +  $  and   $   \times   $
(see also Section \ref{CommutativeSemiringPropertiesII}).

\begin{lemma} \label{LemmaClosureUnderAdditionAndMultiplication}

Let $X$ be an inductive class. 
Then, $ \mathsf{IQ}^{ + } $ proves that there exists an inductive subclass  $Y$ that is closed under $+$ and $ \times $. 

\end{lemma}

\subsection{Ordering  Relation}
\label{LengthRelationLemmaSection}

Let  $ \mathsf{IQ}^{ ++ }  $ be  $ \mathsf{IQ}^{ + }  $ extended with the following axioms 
\[
\forall x \;   [   \   0  \leq  x    \    ]    
,   \     \    \   
\forall x y \;   [    \    x \leq y  \rightarrow   Sx \leq Sy     \       ]   
\]

Using the ideas of Section \ref{SuffixRelationIISection}, 
we  prove the following lemma.

\begin{lemma}  \label{InterMediateQSecondLemma}

 $ \mathsf{IQ}^{ + }   $ and  $ \mathsf{IQ}^{ ++ }  $ are mutually interpretable.  
 
\end{lemma}

\begin{proof}

Since  $ \mathsf{IQ}^{ ++ }  $ is an extension of    $ \mathsf{IQ}^{ + }   $, 
it suffices to show that   $ \mathsf{IQ}^{ ++ }  $ is interpretable in  $ \mathsf{IQ}^{ + }   $.
Furthermore, it suffices to show that we can translate $ \leq $  in such a way that 
$ \mathsf{IQ}^{ + }   $ proves the translation of each instance  of  $  \mathsf{IQ}_3 $ and the translation of  
$ \forall x \;   [   \   0  \leq  x    \    ]     $ and 
$ \forall x y \;   [    \    x \leq y  \rightarrow   Sx \leq Sy     \       ]   $.

Let $ u \in G $ if and only if 
\begin{itemize}

\item[(1)]   $u   \leq    u  $

\item[(2)]   $ \forall w    \leq      u   \;   [     \    w  \leq      w     \    ]    $

\item[(3)]   $ \forall w    \leq      u   \;  \forall  v_0    \leq     w    \;  \forall   v_1    \leq      v_0    \;    [   \   v_1   \leq     w    \    ]   $

\item[(A)]  $   \forall w     \leq     u  \;      [    \  w = 0  \;   \vee   \;  \exists v     \leq   w   \;   [   \  w =  \mathrm{S}  v     \    ]      $

\item[(B)]  $   \forall w     \leq     u  \;   \forall x    \;   [    \   w =   \mathrm{S}  x  \rightarrow  0   \leq   w      \    ]      $

\item[(C)]  $   \forall w     \leq     u  \;   \forall x  y     \;     [    \ 
(       \    w =   \mathrm{S}  y     \;    \wedge    \;        x  \leq   y      \          )     \rightarrow     \mathrm{S}  x  \leq   w      \      ]      $.

\end{itemize} 
It can be verified that $ \mathsf{IQ}  $ proves that     $ t \in  G$ for each variable-free term $t$  and that $G$  is downward closed under  $ \leq $.

We translate $  \leq $ as follows 
\[
x    \leq^{ \tau }   y     \equiv    \   
(                \             y \in G   \;   \wedge   \;   x      \leq      y                \           )  
\    \vee   \  
(         \    y   \not\in G   \;   \wedge   \;       x  = x      \           )  
\      .
\]
Since   $ t \in  G$ for each variable-free term $t$,  the translation of each instance of  the axiom schema $ \mathsf{IQ}_3 $ is a theorem of 
$  \mathsf{IQ}^+   $.

We show that   $   \mathsf{IQ}^+   $ proves the translation of   $   \forall x \;   [   \   0  \leq  x    \    ]     $.
Choose an arbitrary $x $. 
If $ x \not\in G $, then $ 0  \leq ^{ \tau }  x  $ holds  by the second disjunct in the definition of $    \leq^{ \tau }  $. 
Otherwise,   $ x \in G $.  We need to show that $ 0 \leq x $ holds. 
If $ x = 0 $, then $ 0  \leq x $ holds  by  $ \mathsf{IQ}_3 $.
Otherwise, by (A),   there exists $ v \leq x $ such that $ x = \mathrm{S} v $. 
Then, by (B), $ 0 \leq x $ holds. 
Thus,   $    \mathsf{IQ}^+     \vdash    \forall x \;   [   \   0  \leq ^{ \tau }  x    \    ]     $.

We show that    $  \mathsf{IQ}^+  $ proves the translation of   
$  \forall x y \;   [    \    x \leq y  \rightarrow   \mathrm{S} x \leq   \mathrm{S}  y     \       ]     $. 
Assume $ x \leq y $ holds. 
If $  \mathrm{S}  y   \not\in G $, then  $  \mathrm{S}  x  \leq ^{ \tau }      \mathrm{S}  y    $  holds  
 by the second disjunct in the definition of $    \leq^{ \tau }  $. 
 Otherwise,    $  \mathrm{S}  y   \in G $. 
 We need to show that  $ \mathrm{S}  x  \leq \mathrm{S}  y $ holds. 
 By $  \mathsf{Q}_2 $,  $    \mathrm{S}  y   \neq 0 $. 
 Hence, by (A), there exists $ v \leq  \mathrm{S}  y   $ such that  $   \mathrm{S}  y   =    \mathrm{S}  v $. 
 By  $  \mathsf{Q}_1 $,   $ y  = v $.
 Hence, $  y \leq  \mathrm{S}  y  $. 
Since $G$ is downward closed under $ \leq $,  we have $  y \in G $. 
 Then,  by (C),    $ \mathrm{S}  x  \leq \mathrm{S}  y $ holds. 
 Thus,  
 $   \mathsf{IQ}^+  \vdash    \forall x y \;   [    \    x \leq ^{ \tau }  y  \rightarrow   \mathrm{S} x \leq ^{ \tau }    \mathrm{S}  y     \       ]    $.
\end{proof}

\subsection{Interpretation of $ \mathsf{IQ}^{ * }  $ in $ \mathsf{IQ}^{ + }  $ }
\label{InterpretationOfIQStarinIQPlussSection}

\begin{lemma}

$ \mathsf{IQ}^{ * }  $ is  interpretable in $ \mathsf{IQ}^{ +  }  $.  

\end{lemma}

\begin{proof}

By Lemma  \ref{InterMediateQSecondLemma}, 
it suffices to show that  $ \mathsf{IQ}^{ * }  $ is  interpretable in   $ \mathsf{IQ}^{ ++  }  $.  
We interpret   $ \mathsf{IQ}^{ * }  $  in   $ \mathsf{IQ}^{ ++  }  $ by simply restricting the universe of $ \mathsf{IQ}^{ ++  }  $ to an indctive subclass  $K$ that is closed under $+, \times $ and which is such that 
$ \mathsf{IQ}^{ ++  }  $ proves  that   $   \forall x , u  \in K  \;   [    \    u \leq x + u    \    ]     $.

Let 
\[
K_1 = \lbrace u :   \     \forall x \;   [    \    u \leq x + u    \    ]         \    \rbrace 
\]
We have $ 0 \in K_1 $ by the axiom $ \forall x \;   [   \   0  \leq  x    \    ]     $ and $  \mathsf{Q}_4 $.
We show that $K_1$ is closed under  $  \mathrm{S} $.
Let $ u \in K_1 $. 
 We need to show that $ \mathrm{S}  u \in K_1 $. 
That is, we need to show that $ \mathrm{S}  u \leq x +   \mathrm{S}u $. 
Since $ u \in K_1$, we have  $ u \leq x+ u $. 
Then, $ \mathrm{S}  u \leq \mathrm{S} (x+u) $ by  the axiom 
$ \forall x y \;   [    \    x \leq y  \rightarrow   \mathrm{S}  x \leq   \mathrm{S}  y     \       ]    $.
By  $  \mathsf{Q}_5 $,    we have  
\[
  \mathrm{S}  u \leq    \mathrm{S}  (x+u) = x + \mathrm{S}  u 
  \        .
  \]
Hence, $ \mathrm{S} u \in K_1 $. 
Thus, $K_1$ contains $0$ and is closed under    $  \mathrm{S} $. 
By Lemma   \ref{LemmaClosureUnderAdditionAndMultiplication}, 
there exists an inductive subclass $K$ of $K_1 $ that is closed under $  + $ and $  \times $.

We interpret   $ \mathsf{IQ}^{ * }  $  in   $ \mathsf{IQ}^{ ++  }  $ by relativizing quantification to $K$. 
The translation of each one of the axioms 
$ \mathsf{Q}_1   \!  -   \!    \mathsf{Q}_2 $,  $ \mathsf{Q}_4    \!  -   \!    \mathsf{Q}_7 $  
 is a theorem of   $ \mathsf{IQ}^{ ++  }  $  since universal sentences are absolute for $K$. 
It remains to show that each instance of  $ \mathsf{IQ}_3^*  $ is a theorem of  $ \mathsf{IQ}^{ ++  }  $. 
Choose a natural number $n$. 
We need to show that 
\[
\mathsf{IQ}^{ ++  }   \vdash   \forall x ,  y  \in K    \;   [    \   x + y = \overline{n} \rightarrow  
\bigvee_{ k \leq  n  }    y =   \overline{k}       \      ]  
\       .
\]
Assume $ x, y \in K $ and  $ x + y =   \overline{n}    $.
Since $ y \in K \subseteq K_1 $, we have $ y \leq    \overline{n}   $. 
By the axiom schema $  \mathsf{IQ}_3 $, there exists $ k \leq n $ such that $ y =  \overline{k}  $. 
Thus,  $  \mathsf{IQ}^{ ++ } $ proves the translation of each instance of  $ \mathsf{IQ}_3^*  $.
\end{proof}

\section{Interpretability of $   \overline{ \mathsf{ID}   }  $ in   $   \mathsf{IQ}   $ }
\label{InterpretationOfIDInIQ}

\begin{figure}[hbt!]

\[
\begin{array}{r l  c   r l  } 
&{\large \textsf{The Axioms of }   \overline{ \mathsf{ID} }    }
\\
\\
\mathsf{ ID_{1} } 
&  \forall x  y  z   \;   [ \  (x y) z = x (y z)    \ ]  
\\
\mathsf{ ID_{2} } 
&  \forall x y \; [ \  x \neq y \to   ( \ x   0 \neq y   0  \wedge  x   1 \neq y   1 \   )   \ ]  
\\
\mathsf{ ID_{3} } 
&   \forall x y  \;   [    \     x0 \neq y1       \        ]   
\\
\mathsf{ ID_{4} } 
&    \forall x  \;  [  \  x \preceq   \overline{\alpha}    \leftrightarrow  
 \bigvee_{ \gamma  \in \mathsf{Pref}  ( \alpha )   }  x=  \overline{ \gamma}   \   ] 
 \\
\overline{ \mathsf{ID} }_5 
&  \forall x y \; [ \  x \neq y \to   ( \        0 x    \neq 0y     \wedge     1x  \neq   1y    \   )   \ ]  
\\
\overline{ \mathsf{ID} }_6
&   \forall x y  \;   [    \     0x \neq 1y       \        ]   
\\
\\
& {\large \textsf{The Axioms of } \mathsf{IQ} }
\\
\\
 \mathsf{Q}_1   
& \forall x   y \;  [ \  x \neq y \rightarrow \mathrm{S}  x \neq \mathrm{S}  y    \ ]   
\\
 \mathsf{Q}_2  
& \forall x  \;   [    \    \mathrm{S}  x  \neq 0        \         ]   
\\
 \mathsf{Q}_4  
&  \forall x  \;   [ \  x+0 = x  \ ]   
\\
 \mathsf{Q}_5  
& \forall x  y \;   [         \  x+ \mathrm{S}  y = \mathrm{S}  ( x+y)   \        ] 
\\
\mathsf{Q}_6 
&  \forall x  \;   [          \  x \times 0 = 0  \          ]
\\
\mathsf{Q}_7 
&  \forall x  y  \;  [          \  x \times  \mathrm{S}  y =   x \times y    + x   \            ]  
\\
\mathsf{IQ}_3
&
 \forall x \;   [ \  x \leq \overline{n}   \leftrightarrow    \bigvee_{k \leq n } x = \overline{k}   \ ]  
\end{array}
\] 
\caption{
Non-logical axioms of the first-order theories  $ \overline{ \mathsf{ID}  } $ and  $ \mathsf{IQ}$.
}
\label{AxiomsOfIQandIDOverline}
\end{figure}

In this section, we show that  $ \overline{ \mathsf{ID}  } $ is interpretable  in   $   \mathsf{IQ}   $
(see Figure \ref{AxiomsOfIQandIDOverline} for the axioms of  $ \overline{ \mathsf{ID}  } $  and    $   \mathsf{IQ}   $).
The most intuitive way to  interpret concatenation theories in arithmetical theories is to construct    a    formula 
$ \phi_{ \circ } (x, y, z ) $ that given $x $ and $y$  defines an object that  encodes  a  computation  of  $ x  \circ  y   $. 
Unfortunately, 
$   \mathsf{IQ}   $  does not have the resources necessary  to prove that we can find a domain $I$  on which 
 $ \phi_{ \circ } (x, y, z ) $ defines a function 
 that  satisfies  $ \mathsf{ID}_1 $, $ \mathsf{ID}_2 $,  $ \mathsf{ID}_3 $, 
 $ \overline{  \mathsf{ID} }_5    $,  $   \overline{   \mathsf{ID}   }_6    $. 
 To prove correctness of recursive definition  in Robinson Arithmetic $ \mathsf{Q} $, 
 we  rely on  the axiom  
 $ \mathsf{Q}_3  \equiv   \   
\forall x  \;  [   \  x = 0  \;  \vee   \;    \exists y  \;   [    \  x = Sy   \   ]     \    ] 
$.
The axiom schema 
$
\mathsf{IQ}_3 
\equiv       \    
 \forall x  \;    [       \      x \leq \overline{n}   \leftrightarrow            \bigvee_{ k \leq n    }       x=  \overline{ k}         \        ] 
$
can only allow us to verify that  $ \phi_{ \circ } (x, y, z ) $ gives a correct value $z$   when $ x$ and $y $  represent   variable-free terms. 
Thus, to  interpret  $ \overline{ \mathsf{ID}  } $  in   $   \mathsf{IQ}   $, 
we need a conception of  strings as numbers   that allows us to translate concatenation without coding  sequences.
The translation  needs to also be simple enough that we can prove its  correctness   in   $ \mathsf{IQ} $. 
In Lemma  4 of    \cite{Ganea2009},  
Ganea explains  how we can  translate concatenation as a   $ \Delta_0$-formula in strong theories such as  Peano Arithmetic $ \mathsf{PA} $ 
and $ \mathsf{I \Delta}_0 $.

Although we  show that  $ \overline{ \mathsf{ID}  } $ is interpretable  in   $   \mathsf{IQ}   $, 
we have not been able to determine whether the converse holds.

\begin{open problem}

Is  $   \mathsf{IQ}   $    interpretable  in   $\overline{ \mathsf{ID}  }  $?

\end{open problem}

As mentioned, the main result of this section is the following theorem.

\begin{theorem}

 $ \overline{  \mathsf{ID} } $ is  interpretable in $ \mathsf{IQ} $. 

\end{theorem}

The proof of the theorem is structured  as follows: 
In Section  \ref{StringsAsMatrices},  
we explain how we intend to interpret  $ \overline{ \mathsf{ID}  }  $ in  $   \mathsf{IQ}   $. 
In Section \ref{CommutativeSemiringProperties0}, 
we use this idea to give a simple interpretation of $  \mathsf{WD} $ in $  \mathsf{R} $.  
In Section  \ref{CommutativeSemiringPropertiesI},  
we show that we can interpret  $ \overline{ \mathsf{ID}  }     $   
in an extension of  $   \mathsf{IQ}   $ which we denote    $   \mathsf{IQ}^{ (2) }   $.
Finally, in Section \ref{CommutativeSemiringPropertiesII}, 
we show that   $   \mathsf{IQ}   $ and  $   \mathsf{IQ}^{ (2) }   $ are mutually interpretable.

\subsection{Strings as Matrices}
\label{StringsAsMatrices}

The idea  is to think of strings as $2 \times 2 $ matrices and to translate concatenation as matrix multiplication.
Let us first see how we can use this idea to  give a $4$-dimensional interpretation of   
$  (   \lbrace \boldsymbol{0}, \boldsymbol{1}  \rbrace^{ * } ,   \varepsilon , \boldsymbol{0}, \boldsymbol{1}, ^\frown )  $ 
 in  $ (  \mathbb{N} , 0, 1 , + , \times  ) $, 
 where $ \varepsilon $ denotes the  empty string and  
 $  \lbrace \boldsymbol{0}, \boldsymbol{1}  \rbrace^{ * }    
 =  
  \lbrace \boldsymbol{0}, \boldsymbol{1}  \rbrace^{ + }  \cup \lbrace \varepsilon \rbrace $.
Let 
\[
\varepsilon^{ \tau }   := 
  \begin{pmatrix}
1  &  0 
\\
0   &   1   
\end{pmatrix}
, \     \          \   
\         \      \ 
\boldsymbol{0}^{ \tau }   := 
  \begin{pmatrix}
1  &  0 
\\
1   &   1   
\end{pmatrix}
, \     \          \   
\         \      \ 
\boldsymbol{1}^{ \tau }   := 
  \begin{pmatrix}
1  &  1
\\
0   &   1   
\end{pmatrix}
\              .
\]
Let $ \mathbf{SL}_2 (  \mathbb{N} ) $ denote the monoid  generated by  
$  \boldsymbol{0}^{ \tau }  $ and $\boldsymbol{1}^{ \tau }  $ under matrix multiplication. 
The monoid  $\mathbf{SL}_2  (  \mathbb{N} ) $ is a substructure of the special linear group $ \mathbf{SL}_2 ( \mathbb{Z} ) $ 
of $ 2 \times 2 $ matrices with integer coefficients and determinant $1 $. 
Let $ \times $ denote  matrix multiplication. 
Then, 
$  (   \lbrace \boldsymbol{0}, \boldsymbol{1}  \rbrace^{ * } , \varepsilon ,  \boldsymbol{0}, \boldsymbol{1}, ^\frown )  $
is isomorphic to 
$ ( \mathbf{SL}_2  (  \mathbb{N} )  , \varepsilon^{ \tau }  ,  \boldsymbol{0}^{ \tau } ,   \boldsymbol{1}^{ \tau } ,  \times ) $.
Since  $ \mathbf{SL}_2  (  \mathbb{N} )  $ is the set of  $ 2 \times 2 $ matrices with natural number coefficients and determinant $1$, 
the isomorphism defines  a $4$-dimensional interpretation of 
 $  (   \lbrace \boldsymbol{0}, \boldsymbol{1}  \rbrace^{ * } , \varepsilon ,  \boldsymbol{0}, \boldsymbol{1}, ^\frown )  $ 
 in  $ (  \mathbb{N} , 0, 1 , + , \times  ) $.
 The idea is to specify an interpretation of  $ \overline{  \mathsf{ID} } $  in $ \mathsf{IQ} $
 by building on this interpretation of  
 $  (   \lbrace \boldsymbol{0}, \boldsymbol{1}  \rbrace^{ * } , \varepsilon ,  \boldsymbol{0}, \boldsymbol{1}, ^\frown )  $ 
 in  $ (  \mathbb{N} , 0, 1 , + , \times  ) $.
But  we need to be careful since the axioms $ \mathsf{IQ}_1  \! - \! \mathsf{IQ}_2  $,   $ \mathsf{IQ}_4  \! - \! \mathsf{IQ}_7  $ have many models.

In Lemma 11 of    \cite{MurwanashyakaAML},
we use this  idea of associating strings with matrices to prove  that  $   \mathsf{ID}_1  \! -   \!  \mathsf{ID}_3 $ has a decidable model.
We prove this result  by giving  a $4$-dimensional interpretation of  $    \mathsf{ID}_1  \! -   \!  \mathsf{ID}_3 $
in the first-order theory  of the real closed field $ (  \mathbb{R} , 0, 1, + , \times , \leq ) $, 
which is decidable (see Tarski  \cite{tarski1948}).
At the time,
 we were investigating whether it is possible to remove some of the axioms of $ \mathsf{D} $ and obtain a theory that is essentially undecidable. 
The possibility of interpreting $ \overline{ \mathsf{ID}  } $   in   $   \mathsf{IQ}   $
resulted from a careful investigation  of the  algebraic properties  of  $ (  \mathbb{R} , 0, 1, + , \times , \leq ) $ we need to interpret 
$     \mathsf{ID}_1  \! -   \!  \mathsf{ID}_3   $. 
Properties (I)-(VIII)  in  Figure \ref{StringsAsMatricesFigure} are sufficient  to interpret   $ \overline{ \mathsf{ID}  } $   in   $   \mathsf{IQ}   $.
Extending  $   \mathsf{IQ}   $ with    (I)-(VIII)  allows us to  reason about natural numbers in the standard way. 
In  the rest of the paper, we use the Roman numerals (I)-(VIII) to refer exclusively  to  axioms  (I)-(VIII) 
in  Figure \ref{StringsAsMatricesFigure}.

The $4$-dimensional interpretation of 
$  (   \lbrace \boldsymbol{0}, \boldsymbol{1}  \rbrace^{ * } , \varepsilon ,  \boldsymbol{0}, \boldsymbol{1}, ^\frown )  $
 in  $ (  \mathbb{N} , 0, 1 , + , \times   ) $
 we described  is a many-to-one reduction that maps existential sentences to existential sentences. 
This means  that   unsolvability of  equations over 
$  (   \lbrace \boldsymbol{0}, \boldsymbol{1}  \rbrace^{ * } , \varepsilon ,  \boldsymbol{0}, \boldsymbol{1}, ^\frown )  $
 implies unsolvability of  equations over   $ (  \mathbb{N} , 0, 1 , + , \times  ) $.
 The idea of associating 
 $  (   \lbrace \boldsymbol{0}, \boldsymbol{1}  \rbrace^{ * } , \varepsilon ,  \boldsymbol{0}, \boldsymbol{1}, ^\frown )  $
 with  $ \mathbf{SL}_2 (  \mathbb{N} ) $     dates back to A.A. Markov  \cite{Markov1954}.
According to Lothaire \cite{Lothaire2002} (see  p.~387), 
in the 1950s,    A. A. Markov hoped that  Hilbert`s 10th Problem could be solved  by proving unsolvability of word equation, 
that is,  equations over finitely generated free semigroups.
In 1970,   Yuri Matiyasevich  proved  that  Hilbert`s 10th Problem is undecidable  using a completely different method
  (see for example Davis \cite{Davis1973}). 
 In 1977, Makanin  \cite{Makanin1977}     proved that the existential theory of a finitely generated free semigroup is decidable.

\subsection{Interpretation of $  \mathsf{WD}  $ in  $   \mathsf{R}  $}
\label{CommutativeSemiringProperties0}

In this section, we show  that the isomorphism  between 
$   (   \lbrace \boldsymbol{0}, \boldsymbol{1}  \rbrace^{ * } , \varepsilon ,  \boldsymbol{0}, \boldsymbol{1}, ^\frown )  $
and    $ \mathbf{SL}_2 (  \mathbb{N} ) $ 
defines a very simple  interpretation of  $  \mathsf{WD}  $ in  $   \mathsf{R}  $.

\begin{lemma}   \label{SufficientConditionsPart0}

Let $ \tau $ be the  $ 4$-dimensional translation of $ \lbrace 0, 1, \circ  \rbrace $  in 
$ \lbrace 0,   \mathrm{S} , + , \times  \rbrace $ 
defined as follows 
\begin{itemize}

\item[-] $0$ and $ 1 $ are translated as 
$ 
  \begin{pmatrix}
1  &  0 
\\
1   &   1   
\end{pmatrix}
$, 
$
  \begin{pmatrix}
1  &  1
\\
0   &   1   
\end{pmatrix}
$, 
respectively

\item[-] $ \circ $ is translated as matrix multiplication

\item[-] the domain $J$ is the set of all $ 2 \times 2 $ matrices 
$ 
\begin{pmatrix}
x &   y  
\\
z   &    w  
\end{pmatrix}
$ 
where $ x \neq 0 $. 

\end{itemize}
Then,  $ \tau $ extends to a translation of  $ \lbrace 0, 1, \circ  ,  \preceq  \rbrace $  in 
$ \lbrace 0,   \mathrm{S} , + , \times   ,  \leq   \rbrace $ 
that  defines a $ 4$-dimensional  interpretation of   
$  \mathsf{WD}  $  in   $ \mathsf{R} $.

\end{lemma}

\begin{proof}

By the axiom schemas $  \mathsf{R}_1 \equiv   \   \overline{n} + \overline{m} = \overline{n+m}  $, 
$  \mathsf{R}_2 \equiv   \   \overline{n}  \times  \overline{m} = \overline{n \times m}  $,
$\mathsf{R} $ proves the translation of each instance of  
$  \mathsf{WD}_1 \equiv   \    \overline{\alpha} \ \overline{ \beta} = \overline{ \alpha \beta}   $. 
By the axiom schema $  \mathsf{R}_3  $,   $\mathsf{R} $ proves the translation of each instance of    $  \mathsf{WD}_2 $.
It remains to give a  translation of $ \preceq $  that provably satisfies  the axiom schema 
\[
\mathsf{WD}_3 
\equiv      \     
 \forall x [  \  x \preceq     \overline{\alpha}     \leftrightarrow  
 \bigvee_{ \gamma \in   \mathsf{Pref} ( \alpha )    }  x=  \overline{ \gamma}       \          ] 
 \         .
 \]
This is where we use the axiom schema 
$
\mathsf{IQ}_3 
\equiv       \    
 \forall x  \;    [       \      x \leq \overline{n}   \leftrightarrow            \bigvee_{ k \leq n    }       x=  \overline{ k}         \        ] 
$, 
which is a theorem of  $  \mathsf{R} $.

Let 
\[
K = \lbrace 
\begin{pmatrix}
x    &   y  
\\
z   &    w
\end{pmatrix}
:    \    \    
\begin{pmatrix}
x    &   y  
\\
z   &    w
\end{pmatrix}
\neq 
\begin{pmatrix}
1    &   0
\\
0   &    1
\end{pmatrix}
\        \    
\wedge  
\          \ 
xw = 1 + yz
\     \rbrace 
\           .
\]
Let 
\[
A = \begin{pmatrix}
a_1   &   a_2  
\\
a_3  &  a_4
\end{pmatrix}
\     \   
\mbox{  and    }   
\      \   
B = \begin{pmatrix}
b_1   &   b_2  
\\
b_3  &  b_4
\end{pmatrix}
\            .
\]
Let   $ A \preceq B $ if and only if $ A, B \in K $ and 
there exists a largest  element   $ m(B)  \in \lbrace b_1 , b_2, b_3, b_4 \rbrace $  with respect to  $ \leq $    such that 
\begin{itemize}

\item[(1)]   $ A = B $ or

\item[(2)]      there exists  $ C \in K $ such that   $ a_i, c_i \leq m(B)  $ for all  $ 1 \leq i \leq 4 $   and  $ AC = B   \;  $.

\end{itemize}

Let   $  \mathsf{SL}_2 (  \mathbb{N} )^+ $ denote  $ \mathsf{SL}_2 (  \mathbb{N} ) $ minus the identity matrix. 
Assume $B$ is the translation of a variable-free  $ \mathcal{L}_{  \mathsf{BT}  }  $-term.
Then,   $ B \in \mathsf{SL}_2 (  \mathbb{N} )^+  $.
The bound in (2) tells that $ A, C  \in  \mathsf{SL}_2 (  \mathbb{N} )^+ $.
It is straightforward to verify that if $ A, B , C \in   \mathsf{SL}_2 (  \mathbb{N} )^+ $ are such that $ AC = B $, 
then a bound such as the one in (2) holds. 
It is then clear that (1)-(2) capture  what it means for a finite string to be  a prefix of another string. 
Thus,  $ \mathsf{R}  $ proves the translation of each instance of  $  \mathsf{WD}_3 $.
\end{proof}

\subsection{Interpretation of $  \overline{  \mathsf{ID}   } $ in  $   \mathsf{IQ}^{ (2)  }  $}
\label{CommutativeSemiringPropertiesI}

\begin{figure}

\[
\begin{array}{l l l }
\textup{(I)} 
  & 
\mbox{Left  distributivity }    &     \forall x y z   \;   [    \   x(y+z) = xy + xz      \     ]   
\\
\\
\textup{(II)} 
& 
\mbox{Associativity of  } +   &        \forall x y z   \;   [    \  (x+y) + z = x + (y+z)      \     ] 
\\
\\
\textup{(III)} 
& 
\mbox{Associativity of   }  \times  &      \forall x y z   \;   [    \  (xy)  z = x  (yz)      \     ]   
\\
\\
\textup{(IV)} 
 & 
\mbox{Commutativity  of   } +   &       \forall x y    \;   [    \   x+y = y+x      \     ]    
\\
\\
\textup{(V)} 
 & 
\mbox{Commutativity  of   } \times    &    \forall x y    \;   [    \   xy = y x      \     ] 
\\
\\
\textup{(VI)}
& 
\mbox{Right cancellation}  &         \forall x y z \;   [    \   x+ z = y + z \rightarrow   x = y    \     ]   
\\
\\
\textup{(VII)} 
&
\mbox{Nonnegative Elements   }    
&    
 \forall x y  \;   [    \   x+y = 0   \rightarrow     (     \    x = 0       \;  \wedge  \;   y = 0    \     )     \        ]    
\\
\\
\textup{(VIII)} 
&
\mbox{No Zero Divisors  }   &    
 \forall x y  \;   [    \     x y = 0 \rightarrow   (    \   x= 0  \;   \vee   \;   y = 0    \    )           \         ]    
\             .
\end{array}
\]

\caption{
Algebraic properties we need in order  to interpret   $ \overline{  \mathsf{ID} } $  in $ \mathsf{IQ} $. 
}
\label{StringsAsMatricesFigure}
\end{figure}

Let $ \mathsf{IQ}^{ (2) } $ be  $ \mathsf{IQ} $ extended with axioms (I)-(VIII) in  Figure   \ref{StringsAsMatricesFigure}. 
We can reason in  $ \mathsf{IQ}^{ (2) } $  about natural numbers in the standard way 
and will therefore occasionally   not refer explicitly to the axioms of  $ \mathsf{IQ}^{ (2) } $  we use. 
In this section,   we show that  $   \overline{  \mathsf{ID}  }  $  is interpretable  in   $ \mathsf{IQ}^{ (2) } $.

We start by making a few simple observations: 
\begin{itemize}

\item  Axiom (IV) tells us that addition is commutative. 
Hence, by $  \mathsf{Q}_4 $,  $0$ is an additive identity. 
That is, 
$ \mathsf{IQ}^{ (2) }   \vdash       \forall x  \;   [    \   0 + x = x     \;   \wedge   \;   x + 0  = x   \      ]       $.

\item   Recall that  $  1  =    \mathrm{S} 0 $.  
By  $ \mathsf{Q}_7 $  and  $ \mathsf{Q}_6 $
\[
 x 1 =   x  0  + x = 0 + x = x 
 \     .
 \]
 Since axiom (V) tells us that multiplication is commutative, $1$ is a multiplicative identity. 
 That is, 
 $ \mathsf{IQ}^{ (2) }   \vdash     \forall x  \;   [    \   1x = x     \;   \wedge   \;   x1 = x   \      ]     $.

 \item  Axiom (VI) tells us that addition is right-cancellative. 
 Since addition is commutative, it is also left-cancellative. 
  That is
 \[
  \mathsf{IQ}^{ (2) }   \vdash         \forall x y z \;   [    \   z+ x =   z+ y \rightarrow   x = y    \     ]      
  \        .
  \]

\item   By  $  \mathsf{Q}_6 $  and  (V),    
$   \mathsf{IQ}^{ (2) }   \vdash      \forall x   \;   [    \  x0 = 0   \;   \wedge   \;   0x = 0    \     ]    $.

\end{itemize}

\begin{lemma} \label{SufficientConditions}

Let $ \tau $ be the  $ 4$-dimensional translation of $ \lbrace 0, 1, \circ  \rbrace $  in 
$ \lbrace 0,   \mathrm{S} , + , \times  \rbrace $ 
defined as follows 
\begin{itemize}

\item[-] $0$ and $ 1 $ are translated as 
$ 
  \begin{pmatrix}
1  &  0 
\\
1   &   1   
\end{pmatrix}
$, 
$
  \begin{pmatrix}
1  &  1
\\
0   &   1   
\end{pmatrix}
$, 
respectively

\item[-] $ \circ $ is translated as matrix multiplication

\item[-] the domain $J$ is the set of all $ 2 \times 2 $ matrices 
$ 
\begin{pmatrix}
x &   y  
\\
z   &    w  
\end{pmatrix}
$ 
where $ x \neq 0 $. 

\end{itemize}
Then,  $ \tau $ extends to a translation of  $ \lbrace 0, 1, \circ  ,  \preceq  \rbrace $  in 
$ \lbrace 0,   \mathrm{S} , + , \times   ,  \leq   \rbrace $ 
that  defines a $ 4$-dimensional  interpretation of   
$   \overline{  \mathsf{ID}  }  $  in   $ \mathsf{IQ}^{ (2) } $.

\end{lemma}

\begin{proof}

We verify that $J$ satisfies the domain condition.
It is clear that  $ 0^{ \tau }  , 1^{ \tau }    \in J $. 
It remains to verify that $J$ is closed under matrix multiplication. 
Let 
\[
A = \begin{pmatrix}
a_1   &      a_2 
\\
a_3   &   a_4  
\end{pmatrix}
,   \     \      \   
B = \begin{pmatrix}
b_1   &      b_2 
\\
b_3   &   b_4  
\end{pmatrix}
,   \     \      \  
AB =  \begin{pmatrix}
a_1 b_1 + a_2 b_3    &      a_1 b_2 + a_2 b_4  
\\
a_3 b_1 + a_4 b_3    &       a_3 b_2 + a_4 b_4
\end{pmatrix}
\]
where $ a_1, b_1 \neq 0 $. 
We need to show that $ a_1 b_1 + a_2 b_3   \neq 0 $. 
Axiom  (VIII) tells us that models of  $ \mathsf{IQ}^{ (2) } $ do not have zero divisors. 
Hence, $ a_1 b_1 \neq 0 $. 
Axiom (VII) tells us that  $0$  is the only element with  an additive inverse. 
Hence, $ a_1 b_1 + a_2 b_3     \neq 0 $, which implies $ AB \in J $.
Thus,  $J$ is closed under matrix multiplication.

It is straightforward to verify that (I)-(V) suffice  to prove that matrix multiplication is associative. 
Thus,  $ \mathsf{IQ}^{ (2) } $ proves the translation of $ \mathsf{ID}_1 $.

Next, we show that the translation of   $ \mathsf{ID}_2$ and  $  \overline{ \mathsf{ID}  }_5 $ are  theorems of  $ \mathsf{IQ}^{ (2) } $. 
We need to show that 
\begin{enumerate}

\item[(1)]   $  \forall A, B \in J  \;     [          \     
 (   \     A 0^{ \tau } =   B 0^{ \tau }    \;  \vee  \;     0^{ \tau }  A =   0^{ \tau }   B  \    ) 
\rightarrow  A = B    \      ]   
$

\item[(2)]  $  \forall A, B \in J  \;     [          \     
 (   \     A 1^{ \tau } =   B 1^{ \tau }    \;  \vee  \;     1^{ \tau }  A =   1^{ \tau }   B  \    ) 
\rightarrow  A = B    \      ]   
$.

\end{enumerate}
We  verify (1). 
First, we show that $  \forall A, B \in J  \;     [          \      A 0^{ \tau } =   B 0^{ \tau }     \rightarrow  A = B    \      ]   $. 
Assume  $ x, a   \neq 0  $ and 
\[
\begin{pmatrix}
x+ y   & y   
\\
z +w    &  w 
\end{pmatrix}
= 
\begin{pmatrix}
x  & y   
\\
z   &  w 
\end{pmatrix}
 0^{ \tau }
= 
\begin{pmatrix}
a  & b   
\\
c   &  d
\end{pmatrix}
 0^{ \tau }
 = 
 \begin{pmatrix}
a + b   & b   
\\
c + d    &  d
\end{pmatrix}
\       .
\]
We need to show that $ x = a$ and $ z = c $.
We have 
\[
x+ b  = x + y = a+ b  
\    \wedge     \  
z+d =  z + w =  c +d 
\     .
\]
Since addition is right-cancellative,   $ x = a $ and $ z = c $.
Thus,  for all  $   A, B \in J   $,    if  $  A 0^{ \tau } =   B 0^{ \tau }    $,   then $  A = B       $.

We show that   $  \forall A, B \in J  \;     [          \         0^{ \tau }  A =   0^{ \tau }   B       \rightarrow  A = B    \      ]     $. 
Assume $ x, a  \neq 0 $ and 
\[
\begin{pmatrix}
x    &      y   
\\
x + z     &    y +  w 
\end{pmatrix}
=
 0^{ \tau }
\begin{pmatrix}
x  & y   
\\
z   &  w 
\end{pmatrix}
= 
 0^{ \tau }
\begin{pmatrix}
a  & b   
\\
c   &  d
\end{pmatrix}
= 
\begin{pmatrix}
a   &     b
\\
a+c     &  b+d
\end{pmatrix}
\          .
\]
We need to show that $ z = c $ and $ w = d $.
We have 
\[
a + z = x + z   =   a + c 
\     \wedge    \  
b + w  = y + w   = b + d 
\      .
\]
Since addition is left-cancellative,   $ z = c $ and $ w = d $. 
Thus,  for all $  A, B \in J $, if   $   0^{ \tau }  A =   0^{ \tau }   B      $,   then $ A = B $. 
Hence, (1) holds. 
By similar reasoning, (2) holds. 
Thus, $ \mathsf{IQ}^{ (2) } $ proves the translation of   $ \mathsf{ID}_2$ and  $  \overline{ \mathsf{ID}  }_5 $.

We show that the translation of   $ \mathsf{ID}_3$ is a theorem of  $ \mathsf{IQ}^{ (2) } $. 
We need to show that    $  \forall A, B \in J  \;     [          \        A 0^{ \tau }  \neq    B 1^{ \tau }      \      ]      $.
Assume for the sake of a contradiction $ x,  a \neq 0 $ and 
\[
\begin{pmatrix}
x+ y   & y   
\\
z +w    &  w 
\end{pmatrix}
= 
\begin{pmatrix}
x  & y   
\\
z   &  w 
\end{pmatrix}
 0^{ \tau }
= 
\begin{pmatrix}
a  & b   
\\
c   &  d
\end{pmatrix}
 1^{ \tau }
= 
\begin{pmatrix}
a   &   a + b   
\\
c   &  c + d   
\end{pmatrix}
\          .
\]
Then 
\[
a  = x + y   = x + a + b 
\]
where we have omitted parentheses since addition is associative. 
Since $0$ is an additive identity and addition is commutative, 
$ 0 + a = x + b + a  $. 
Since addition is right-cancellative,   $ 0 = x + b $. 
Since $0$ is the only element with an addititive inverse, $ x = 0 $, which contradicts the assumption that $ x \neq 0 $. 
Thus,  $ \mathsf{IQ}^{ (2) } $ proves the translation of   $ \mathsf{ID}_3$.

We show that the translation of     $  \overline{ \mathsf{ID}  }_6  $ is a   theorem of  $ \mathsf{IQ}^{ (2) } $. 
We need to show that   $  \forall A, B \in J  \;     [          \     0^{ \tau }  A   \neq     1^{ \tau }   B  \      \      ]   $.
Assume for the sake of a contradiction $ x,  a \neq 0 $ and 
\[
\begin{pmatrix}
x   &    y   
\\
x+z    &  y+w 
\end{pmatrix}
= 
 0^{ \tau }
\begin{pmatrix}
x  & y   
\\
z   &  w 
\end{pmatrix}
= 
 1^{ \tau }
\begin{pmatrix}
a  & b   
\\
c   &  d
\end{pmatrix}
= 
\begin{pmatrix}
a+c   &   b + d   
\\
c   &     d   
\end{pmatrix}
\        .
\]
Then, $ x = a+c    = a + x + z  $.
Hence, $ 0 = a + z $. 
Since $0$ is the only element with an addititive inverse,    $ a = 0 $,    which contradicts the assumption that $ a \neq 0 $. 
Thus,  $ \mathsf{IQ}^{ (2) } $ proves the translation of   $     \overline{ \mathsf{ID}  }_6    $.

Finally, we translate $ \preceq $ as in the proof of Lemma \ref{SufficientConditionsPart0}.
\end{proof}

\section{ Interpretation of  $ \mathsf{TC}   $    in   $ \mathsf{Q} $  }
\label{RecursionFreeInterpretationOfTCinQ}

\begin{figure}
\[
\begin{array}{r l }
&{\large \textsf{The Axioms of } \mathsf{ TC }^{ \varepsilon }       }
\\
\\
\mathsf{TC }_1^{ \varepsilon }   
& 
\forall x  \;  [       \      \varepsilon   x = x   \;   \wedge   \;   x  \varepsilon  =   x   \        ]  
\\
\mathsf{TC }_2^{ \varepsilon }   
& 
\forall x y z \;  [    \   x(yz) = (xy)z   \     ]  
\\
  \mathsf{TC }_3^{ \varepsilon }  
&
\forall x y z w  \;    [          \        \  xy = zw   \rightarrow      \exists u  \;   [      \    
 (     \        z= xu  \wedge uw = y                 \              ) \vee 
 \\ 
 & 
 \qquad    \qquad     \qquad      \qquad    \qquad      \qquad  
(         \            x = zu \wedge uy = w  \               )                \                   ]           \         ] 
\\
 \mathsf{ TC }_4^{ \varepsilon }    
&
 0 \neq \varepsilon 
\\
 \mathsf{ TC }_5^{ \varepsilon }  
&
\forall x y  \;         [        \           xy = 0    \rightarrow    (     \    x = \varepsilon  \;   \vee  \;  y = \varepsilon    \     )      \         ] 
\\
 \mathsf{ TC }_6^{ \varepsilon }  
&
 1 \neq \varepsilon 
\\
 \mathsf{ TC }_7^{ \varepsilon }  
&
\forall x y  \;         [        \           xy = 1    \rightarrow    (     \    x = \varepsilon  \;   \vee  \;  y = \varepsilon    \     )      \         ] 
\\
 \mathsf{ TC }_8^{ \varepsilon }   
&
0 \neq 1
\end{array}
\]

\caption{
Non-logical axioms of the first-order theory   $ \mathsf{TC}^{ \varepsilon }    $.
}
\label{AxiomsOfTCWithIdentityElement}
\end{figure}

In this section, we show that  our interpretation of  $   \overline{ \mathsf{ID}   }  $ in   $   \mathsf{IQ}   $  extends in a natural way  to an interpretation of   $ \mathsf{TC} $ in   $   \mathsf{Q}   $.
Instead of interpreting      $ \mathsf{TC} $, 
we interpret the variant   $ \mathsf{TC}^{ \varepsilon }  $  where we extend the language of  $ \mathsf{TC} $ with a constant symbol 
$ \varepsilon $ for the identity element. 
See Figure  \ref{AxiomsOfTCWithIdentityElement} for the axioms of    $ \mathsf{TC}^{ \varepsilon }  $.
We choose to work with $ \mathsf{TC}^{ \varepsilon }  $ because  the identity matrix  is naturally present in our interpretation of 
$   \overline{ \mathsf{ID}   }  $ in   $   \mathsf{IQ}   $ 
and because we get a more compact form of the editor axiom
($ \mathsf{TC}_2 $  and  $ \mathsf{TC}_3^{ \varepsilon } $). 
The interpretation we give can be turned into an interpretation of  $ \mathsf{TC} $  by simply removing the identity matrix from the domain 
(see Appendix A of Visser  \cite{visser} for  mutual interpretability of  $ \mathsf{TC} $ and $ \mathsf{TC}^{ \varepsilon }  $).

Recall that  $ x \leq_{ \mathsf{l} }   y  \equiv   \   \exists  r  \;   [   \   r + x = y      \   ]   $. 
Let $   x <_{ \mathsf{l} }   y  \equiv   \   \exists  r  \;   [     \   r \neq 0  \;  \wedge   \;    r + x = y      \       ]    $.
Let $ \mathsf{Q}^{ (2) } $ be $ \mathsf{Q} $ extended with axioms  (I)-(VI) in  Figure   \ref{StringsAsMatricesFigure} and  the trichotomy law 
\[
\forall x y  \;   [    \     x  <_{ \mathsf{l} }  y   \;   \vee   \;      x = y       \;   \vee   \;   y  <_{ \mathsf{l} }  x    \      ] 
\       .
\]
We make a few simple observations: 
\begin{itemize}

\item[-] Axiom (VII)    $ \forall x y  \;   [    \   x+y = 0   \rightarrow     (     \    x = 0       \;  \wedge  \;   y = 0    \     )     \        ]   $
is a theorem of  $ \mathsf{Q}^{ (2) } $. 
Indeed, assume $ x+y = 0 $. 
If $ y = 0 $, then $ x = 0 $ by $  \mathsf{Q}_4 $. 
Thus, it suffices to show that $ y = 0 $. 
Assume for   the sake of a contradiction  that $ y \neq 0 $. 
Then, by  $  \mathsf{Q}_3 $, there exists $ v $ such that $ y = \mathrm{S} v $. 
By $  \mathsf{Q}_5 $
\[
 0 = x+y   = x +  \mathrm{S} v =  \mathrm{S} ( x+v ) 
 \]
 which contradicts $ \mathsf{Q}_2 $. 
 Thus,  $ x+y = 0 $ implies $ x =y = 0 $.

\item[-] Axiom (VIII)   $  \forall x y  \;   [    \     x y = 0 \rightarrow   (    \   x= 0  \;   \vee   \;   y = 0    \    )           \         ]   $
is a theorem of  $ \mathsf{Q}^{ (2) } $. 
Indeed, assume  $ xy = 0$  and $  y \neq 0 $. 
By $  \mathsf{Q}_3 $, there exists  $v $ such that  $ y = \mathrm{S} v $. 
By $  \mathsf{Q}_7 $
\[
0 = xy = x v + x    
\]
which implies $ x = 0$. 
Thus, $ xy = 0 $ implies $ x= 0  \;  \vee  \;  y = 0 $.

\item[-]  $  \mathsf{Q}^{ (2) } $ proves that   $1 $ is the only element with a multiplicative inverse. 
Indeed, assume  $ xy = 1$. 
By commutativity of multiplication,  $  \mathsf{Q}_2 $ and $  \mathsf{Q}_5 $, we have $ x, y  \neq  0 $. 
Hence, by  $  \mathsf{Q}_3 $,  there  exist $ u, v $ such that   $ x = \mathrm{S} u $  and  $ y = \mathrm{S} v $. 
By commutativity of multiplication,   $  \mathsf{Q}_7 $  and $  \mathsf{Q}_5 $
\[
1 =  x y =   x v + x  =  \mathsf{S} (xv + u) 
\     \wedge     \   
1 =  yx  =   yu + y   =  \mathsf{S} ( yu  + v) 
\       .
\]
By  $  \mathsf{Q}_1 $
\[
0 =  xv + u
\     \wedge     \   
0  =  yu  + v 
\]
which implies $ u = v = 0 $. 
Hence, $ x = y = 1 $. 
Thus, $ xy = 1 $ implies $ x = y = 1 $.

\end{itemize}

This section as structured as follows: 
In Section  \ref{PredecessorLemmaSection}, 
we show that  that if we modify our interpretation of  $ \overline{ \mathsf{ID} } $ in  $ \mathsf{IQ}^{ (2) } $  by choosing as the domain the class $K$ of all $ 2 \times  2$ matrices with determinant one, we obtain an interpretation in  $ \mathsf{Q}^{ (2) } $ of the theory we obtain from 
$  \mathsf{TC}^{ \varepsilon } $ 
by replacing the editor axiom   $  \mathsf{TC}^{ \varepsilon }_2 $  with the axioms 
$  \mathsf{D}_2 $,  $  \mathsf{D}_3 $,   $  \overline{  \mathsf{ID} }_5   $,  $  \overline{  \mathsf{ID} }_6   $, 
 $   \forall x  \;   [    \   x =  \varepsilon   \;   \vee   \;    \exists  y  \;   [     \    x = y0   \;   \vee   \;    x = y1    \      ]    $.
In  Section  \ref{InterpretationOfTCInQ2},
we extend our interpretation of $ \overline{ \mathsf{ID} } $ in  $ \mathsf{IQ}^{ (2) } $ 
to an interpretation of $ \mathsf{TC}^{ \varepsilon  } $ in   $ \mathsf{Q}^{ (2) } $   
by restricting $K$  to a subclass on which the editor axiom holds. 
Finally, in Section  \ref{CommutativeSemiringPropertiesIII}, 
we show that we can interpret $  \mathsf{Q}^{ (2) } $ in $ \mathsf{Q}  $ by  restricting the universe of   $ \mathsf{Q} $ to a suitable subclass.

\subsection{Atoms and Predecessors }
\label{PredecessorLemmaSection}

Let $K$ denote the class  all $ 2 \times 2 $ matrices with determinant $ 1$. 
That is 
\[
K = \lbrace    \begin{pmatrix}
x  &  y
\\
z   &   w
\end{pmatrix}     :  
\     xw = 1 + yz   
\     \rbrace 
\         .
\]

It is not difficult to verify that  $ \mathsf{Q}^{ (2) } $ proves that $ \det (AB) = 1 $ if  $ \det (A) = \det (B) = 1 $.
We thus have the following lemma.

\begin{lemma}\label{InterpreationOfTCLemmaDomainInitialClass}

$  \mathsf{Q}^{ (2) } $ proves that   $K$ is closed under 
$ 
  \begin{pmatrix}
1  &  0 
\\
0   &   1   
\end{pmatrix}
$,
$ 
  \begin{pmatrix}
1  &  0 
\\
1   &   1   
\end{pmatrix}
$,  
$
  \begin{pmatrix}
1  &  1
\\
0   &   1   
\end{pmatrix}
$
and matrix multiplication.

\end{lemma}

Let us  say that $ A \in K $ is an atom in  $K$ if for all $ B, C \in K $,   $ A = BC $ implies that one of $B$ and $C$ is the identity matrix. 
The proof of the following lemma is straightforward.

\begin{lemma}  \label{InterpreationOfTCLemmaAtoms}

$ \mathsf{Q}^{ (2) }    $ proves that 
$ 
  \begin{pmatrix}
1  &  0 
\\
0   &   1   
\end{pmatrix}
$, 
 $ 
  \begin{pmatrix}
1  &  0 
\\
1   &   1   
\end{pmatrix}
$
and 
$
  \begin{pmatrix}
1  &  1
\\
0   &   1   
\end{pmatrix}
$
are atoms in $K$. 

\end{lemma}

In   \cite{MurwanashyakaAML}, we introduce a theory  $  \mathsf{BTQ}  $  and show that  $  \mathsf{D} $ interprets $  \mathsf{Q} $ by showing that it interprets $  \mathsf{BTQ} $.
We obtain   $  \mathsf{BTQ}  $ from $  \mathsf{ID} $ by replacing  $  \mathsf{ID}_4 $ with the axiom 
$   \forall x  \;   [    \   x = 0  \;   \vee   \;  x = 1     \;   \vee   \;   \exists  y  \;   [     \    x = y0   \;   \vee   \;    x = y1    \      ]    $.
The next lemma shows that  if we modify  the translation in Lemma \ref{SufficientConditions}  by choosing as the domain
the class of  all elements in $K$ distinct from the identity matrix, 
we obtain  an interpretation of  $  \mathsf{BTQ}  $ in   $ \mathsf{Q}^{ (2) }    $.

\begin{lemma} \label{InterpreationOfTCLemmaPredecessors}

Let  $ A \in K  $.
Then,  $ \mathsf{Q}^{ (2) }    $ proves that 
$A$ is the identity matrix or 
  that there exist   $ B, C  \in K $ such that 
$ A = BC $ and $C$ is one of   
$ 
  \begin{pmatrix}
1  &  0 
\\
1   &   1   
\end{pmatrix}
$,  
$
  \begin{pmatrix}
1  &  1
\\
0   &   1   
\end{pmatrix}
$.

\end{lemma}

\begin{proof}

Let $ K  \ni  A = \begin{pmatrix}
a  &   b 
\\
c  &   d   
\end{pmatrix}
$.
By the trichotomy law, we have the following cases
\begin{itemize}

\item[-]  (1a)   $ a = b \;   \wedge   \;   c   <_{ \mathsf{l} }    d $,  
(1b)    $  a = b    \;   \wedge   \;     c = d $, 
(1c)     $ a = b      \;   \wedge   \;    d    <_{ \mathsf{l} }   c $

\item[-]   (2a)  $  a <_{ \mathsf{l} }  b    \;   \wedge   \;     c = d $, 
(2b)     $ b <_{ \mathsf{l} }  a   \;   \wedge   \;      c = d $

\item[-]  (3a)   $ b <_{ \mathsf{l} } a    \;   \wedge   \;        c <_{ \mathsf{l} }  d  $, 
(3b)    $   a <_{ \mathsf{l} }  b     \;   \wedge   \;        d <_{ \mathsf{l} } c  $, 
(3c)      $  b <_{ \mathsf{l} }  a       \;   \wedge   \;      d <_{ \mathsf{l} }  c   $, 
(3d)   $ a <_{ \mathsf{l} }  b    \;   \wedge   \;            c <_{ \mathsf{l} }  d         \;        $.

\end{itemize}

We consider   Case  (1a).
Since  $ a = b \;   \wedge   \;   c   <_{ \mathsf{l} }    d $,    let  $ d = r + c $ where $ r \neq 0 $. 
Since $  ad = 1 + bc  $  as $  A  \in K $,    we have  
\[
ar +  ac =  a (r+c)  = ad = 1 + bc   = 1 + ac
 \]
Since addition is right-cancellative,    $ ar = 1 $,    which implies $ a = r =  1$. 
 Thus
 \[
  A = \begin{pmatrix}
 1    &   1
 \\
 c   &     1+c 
 \end{pmatrix}
 = 
  \begin{pmatrix}
 1    &   0
 \\
 c   &     1 
 \end{pmatrix}
  \begin{pmatrix}
 1    &   1
 \\
 0   &     1
 \end{pmatrix}
 \         .
\]

We consider Case (1b). 
Since  $  a = b    \;   \wedge   \;     c = d $ and $ A \in K $,  we have 
\[
 ad = 1+  bc  = 1+ ad
 \       .
 \]
 Since $0$ is an additive identity and addition is right-cancellative,    $ 0 = 1 $   which contradicts $ \mathsf{Q}_2     $.

We consider Case (1c). 
Since   $ a = b      \;   \wedge   \;    d    <_{ \mathsf{l} }   c $,    let  $ c = s + d $ where $ s \neq  0 $.
We have   
\[
 ad = 1 + bc   = 1 + a (s+d ) = 1 + as + ad  
 \          .
 \]
Hence, $ 0 = 1+ as $. 
Since addition is commutative, $ 0 = \mathrm{S} ( as + 0 )  $  by $  \mathsf{Q}_5 $,     which contradicts $ \mathsf{Q}_2     $.

We consider Case (2a).
Since    $  a <_{ \mathsf{l} }  b    \;   \wedge   \;     c = d $,    let    $ b = r + a $ where $ r \neq 0 $. 
We have 
\[
ad = 1+ bc  = 1 +  (r+a) d  = 1 +  rd + ad 
\        .
\]
Hence, $ 0  = 1 + rd $ which contradicts $ \mathsf{Q}_2    $.

We consider Case (2b).
Since      $ b <_{ \mathsf{l} }  a   \;   \wedge   \;      c = d $,   let    $ a = s + b $ where $ s \neq 0 $.
We have 
\[
sd + bd  = 
(s+b) d  =
ad = 1 +bc   = 1 + bd 
\        .
\]
Hence, $ sd = 1 $ which implies $ s = d = 1 $. 
Thus
\[
A = \begin{pmatrix}
 1 +b    &   b
 \\
 1   &     1
 \end{pmatrix}
 = 
 \begin{pmatrix}
 1   &   b
 \\
 0   &     1
 \end{pmatrix}
 \begin{pmatrix}
 1     &   0
 \\
 1   &     1
 \end{pmatrix}
 \        .
\]

We consider Case (3a). 
Since   $ b <_{ \mathsf{l} } a    \;   \wedge   \;        c <_{ \mathsf{l} }  d  $, 
there exist $ r , s \neq 0 $ such that $ a = r + b $ and $ d = s + c $. 
Since $ ad = 1 + bc $,    we have 
\[
rs + rc + bs + bc =  (r+b) ( s+c)  = ad = 1 + bc  
\]
which implies   
\[
  rs + rc + bs  = 1   
  \        .
  \]
Since $ r, s  \neq 0 $,  we conclude that   $ r = s = 1 $ and $  b = c = 0 $. 
Thus,  $  A = \begin{pmatrix}
1  &    0  
\\
0    &  1  
\end{pmatrix}
$.

We consider Case (3b). 
Since  $   a <_{ \mathsf{l} }  b     \;   \wedge   \;        d <_{ \mathsf{l} } c  $, 
there exist  $ p, q \neq 0 $ such that    $ b = p + a $ and $ c = q + d $.   
Since $ ad = 1 + bc $,    we have 
\[
ad = 1 + bc  = 1 +  (p+a) (q+d)  
= 1 + pq + pd + aq + ad
\          .
\]
Hence,  $ 0 =   1 + pq + pd + aq $  which contradicts $ \mathsf{Q}_2     $.

We consider Case  (3c).
Since  $  b <_{ \mathsf{l} }  a       \;   \wedge   \;      d <_{ \mathsf{l} }  c   $, 
\[
A =   \begin{pmatrix}
a  &   b 
\\
c  &   d   
\end{pmatrix}
= 
E 
\begin{pmatrix}
1  &   0
\\
1  &   1 
\end{pmatrix}
\      \       \     
\mbox{  where }
E =    \begin{pmatrix}
a-b  &   b 
\\
c-d  &   d   
\end{pmatrix}   
\         .
\]
Since addition is right-cancellative in $ \mathsf{Q}^{ (2) } $, 
we write $ a-b $ and $ c-d $  for the unique elements $r, s $ such that $ a = r + b $ and $ c = s + d $.
We need to show that $ E   \in K   $. 
That is, we need to show that   $  \det (E ) = 1 $. 
First, observe that 
\[
 (x-y) z = xy - yz 
  \  
   \mbox{ and  }   
     \  
      ( 1 + xz ) - y z  = 1 + ( xz - yz) 
     \] 
  since 
\[
 (x-y) z  + yz  =    \big(   (x-y) + y  \big)  z    =  x z 
 \  
   \mbox{ and  }   
     \  
 1 + ( xz - yz ) + yz = 1 + xz 
 \        .
\]
Since $ \det (A) = 1 $,  we have 
\[
(a-b) d = ad - bd = ( 1+ bc)  -bd  = 1 + b ( c-d) 
\       .
\]
Thus,   $  E \in K $.

We consider  Case (3d). 
Since   $ a <_{ \mathsf{l} }  b    \;   \wedge   \;            c <_{ \mathsf{l} }  d          $
\[
A =   \begin{pmatrix}
a  &   b 
\\
c  &   d   
\end{pmatrix}
= 
G
\begin{pmatrix}
1  &   1
\\
0  &   1 
\end{pmatrix}
\      \       \     
\mbox{  where }
G  =   \begin{pmatrix}
a  &   b -a
\\
c  &   d  -c
\end{pmatrix}
\]
We need to show that $ G  \in K   $. 
That is, we need to show that   $ \det (G) = 1 $.
Since $ \det (A) = 1 $, we have 
\[
a (d-c) = ad - ac  = 1 + bc - ac = 1 + (b-a) c 
\       .
\]
Thus,  $G \in K $.
\end{proof}

\subsection{Interpretation of $  \mathsf{TC}^{ \varepsilon }   $ }
\label{InterpretationOfTCInQ2}

We are finally ready to  extend our interpretation of  $ \overline{  \mathsf{ID} } $ in $  \mathsf{IQ}^{ (2) } $ 
to an interpretation of  $  \mathsf{TC}^{ \varepsilon  }   $ in   $  \mathsf{Q}^{ (2) } $.
All we need to do is to restrict the class $K$  to a subclass on which  the editor axiom holds.

\begin{theorem}

There exists a class $I$ such that the the $ 4$-dimensional translation of 
$ \lbrace   \varepsilon , 0 , 1, \circ  \rbrace $  in $ \lbrace 0, 1, S, + , \times  \rbrace $ 
defined by 
\begin{itemize}

\item[-] $ \varepsilon $,  $0$ and $ 1 $ are translated as 
$ 
  \begin{pmatrix}
1  &  0 
\\
0   &   1   
\end{pmatrix}
$, 
$ 
  \begin{pmatrix}
1  &  0 
\\
1   &   1   
\end{pmatrix}
$, 
$
  \begin{pmatrix}
1  &  1
\\
0   &   1   
\end{pmatrix}
$, 
respectively

\item[-] $ \circ $ is translated as matrix multiplication

\item[-] the domain   is $I$ 
\end{itemize}
 defines a $ 4$-dimensional  interpretation of   
$  \mathsf{TC}^{ \varepsilon }   $  in   $ \mathsf{Q}^{ (2) } $.

\end{theorem}

\begin{proof}

Let
\[
K = \lbrace    \begin{pmatrix}
x  &  y
\\
z   &   w
\end{pmatrix}     :  
\     xw = 1 + yz   
\     \rbrace 
\         .
\]
Lemma  \ref{SufficientConditions} and  Lemma \ref{InterpreationOfTCLemmaAtoms}  tell us that the  restriction of axioms 
 $  \mathsf{TC}^{ \varepsilon }_1 \! -    \!      \mathsf{TC}^{ \varepsilon }_2 $,  
$  \mathsf{TC}^{ \varepsilon }_4 \! -    \!      \mathsf{TC}^{ \varepsilon }_8 $
to $K$  are theorems of   $ \mathsf{Q}^{ (2) } $. 
Since   $  \mathsf{TC}^{ \varepsilon }_1 \! -    \!      \mathsf{TC}^{ \varepsilon }_2 $,  
$  \mathsf{TC}^{ \varepsilon }_4 \! -    \!      \mathsf{TC}^{ \varepsilon }_8 $ 
are universal sentences, 
to interpret    $  \mathsf{TC}^{ \varepsilon } $  in $   \mathsf{Q}^{ (2) }    $, 
it suffices to restrict the  class $K$ to a subclass $I$ on which the editor axiom  $  \mathsf{TC}^{ \varepsilon }_2 $ holds. 
We need the following three  properties that are given by 
 Lemma  \ref{SufficientConditions} and   Lemma    \ref{InterpreationOfTCLemmaPredecessors} 
 \[
 \begin{array}{r  l }
   \mathsf{DJ}
  &   
   \forall A , B \in K   \;      \Big[     \   
    A      \begin{pmatrix}
1  &  0 
\\
1   &   1   
\end{pmatrix}     
\neq   
B   \begin{pmatrix}
1      &      1
\\
0    &   1   
\end{pmatrix}      
\           \Big] 
\\
   \\
  \mathsf{RC}
  &   
   \forall A , B \in K   \;      \Big[     \   A  \neq B  \rightarrow  
   \\
   & 
   \quad    \quad
(    \    A      \begin{pmatrix}
1  &  0 
\\
1   &   1   
\end{pmatrix}     
\neq   
B   \begin{pmatrix}
1  &  0 
\\
1   &   1   
\end{pmatrix}      
\      \wedge     \   
A     \begin{pmatrix}
1  &  1
\\
0   &   1   
\end{pmatrix} 
\neq 
B    \begin{pmatrix}
1  &  1
\\
0   &   1   
\end{pmatrix}
\      )         \         \Big]
\\
\\
 \mathsf{PD} 
 &  
 \forall A  \in K   \;    \Big[    \    
 A  =       \begin{pmatrix}
1  &  0 
\\
0   &   1   
\end{pmatrix}
\    \vee    \  
\\  
&    
\quad  \quad  
\exists B  \in K   \;   \big[      \   
A = B      \begin{pmatrix}
1  &  0 
\\
1   &   1   
\end{pmatrix}
\     \vee     \  
A   =   B       \begin{pmatrix}
1  &  1
\\
0   &   1   
\end{pmatrix}
\         \big]           \        \Big]  
 \end{array}
 \]

Let 
\begin{multline*}
H =  \lbrace   W   \in K   :    \       \forall X  Z     \;     \forall Y   \in  K     \big[          \  
      \  XY = ZW   \rightarrow      \exists U  \in K   \;   [      \    
\\
 (     \        Z = XU       \;   \wedge    \;        UW = Y                \              ) \vee 
(         \            X = ZU    \;   \wedge    \;      UY = W            \               )                \                   ]           \         \big] 
\     \rbrace 
\        .
\end{multline*}
It follows from $  \mathsf{DJ} $, $  \mathsf{RC} $,  $ \mathsf{PD} $ and associativity of matrix multiplication  that 
$ 
  \begin{pmatrix}
1  &  0 
\\
0   &   1   
\end{pmatrix}
$, 
$ 
  \begin{pmatrix}
1  &  0 
\\
1   &   1   
\end{pmatrix}
$, 
$
  \begin{pmatrix}
1  &  1
\\
0   &   1   
\end{pmatrix}
$
are elements of $H$. 
We show that $H$ is closed under matrix multiplication. 
So, assume $ W_0 , W_1 \in H $. 
We need to show that $ W_0 W_1 \in H $. 
First, we observe that $ W_0 W_1 \in K $ since   $K$ is closed under matrix multiplication and $ H \subseteq K $.
Now, let $ X, Y, Z $ be such that $ XY = Z W_0 W_1 $ and $ Y \in K $. 
Since $ W_1 \in H $, we have the following two cases for some $ U_1 \in K $
\[
(1)    \        X = Z W_0 U_1   \;   \wedge   \;   U_1 Y = W_1
\;    ,   \      \      \     
(2)    \      Z W_0  = X U_1  \;   \wedge   \;   U_1 W_1  = Y  
\        .
\]
We consider (1). 
Since $K$ is closed under matrix multiplication and $ H \subseteq K $,   we have  
\[
X = Z W_0 U_1            \   \wedge   \   W_0 U_1 Y =   W_0 W_1 
\           \wedge        \     
W_0 U_1  \in K 
\        .
\tag{*}
\]
We consider (2). 
Since $ W_0 \in  H $ and $ U_1  \in K $, we have one of the following two cases for some $ U_0 \in K $
\begin{multline*}
\textup{ (2a) }   \    Z = X U_0   \;   \wedge   \;  U_0 W_0  = U_1       \;   \wedge   \;   U_1 W_1  = Y  
,    \      \       \    
\\
\textup{ (2b) }  \      X = Z U_0   \;   \wedge   \;  U_0 U_1 = W_0      \;   \wedge   \;   U_1 W_1  = Y  
\         .
\end{multline*}
In case of (2a), we have 
\[
Z = X U_0      \   \wedge   \       U_0 W_0 W_1  =    U_1  W_1    =  Y 
\    \wedge    \   
U_0   \in K 
\        .
\tag{**}
\]
In case of (2b), we have 
\[
X =   Z U_0   \   \wedge   \   W_0  W_1   =     U_0 U_1  W_1  =  U_0  Y  
\    \wedge    \   
U_0   \in K 
\        .
\tag{***}
\]
By (*),  (**) and (***), we have $ W_0  W_1 \in H $.
Thus,  $H$ is closed under matrix multiplication.

We are finally ready to specify the class $I$. 
Let 
\[
I = \lbrace A  \in H  :   \   \forall   B  \;   [    \   B    \preceq_{ K }   A    \rightarrow   B \in H     \     ]  
\    \rbrace
\]
where   
\[
B  \preceq_{ K  }   A     \equiv   \     \exists  C  \in K   \;   [     \     A =  B C      \      ]      
\             .
\]
It follows from  Lemma    \ref{InterpreationOfTCLemmaAtoms}   that 
$ 
  \begin{pmatrix}
1  &  0 
\\
0   &   1   
\end{pmatrix}
$, 
$ 
  \begin{pmatrix}
1  &  0 
\\
1   &   1   
\end{pmatrix}
$, 
$
  \begin{pmatrix}
1  &  1
\\
0   &   1   
\end{pmatrix}
$
are elements of $ I $. 
To show that $I$ defines  a model of $   \mathsf{TC}^{ \varepsilon } $, 
 it suffices to show that $ I$ is closed under  matrix  multiplication  and downward closed under  $ \preceq_{ K  }  $, 
 where the latter ensures that the editor axiom holds restricted to $I$.

We show that $I$ is closed under matrix multiplication.
Assume $ A_0 , A_1 \in I $. 
We need to show that $ A_0 A_1 \in I $. 
So, assume $ B C = A_0 A_1 $  where  $ C \in K $. 
We need to show that $ B \in H $. 
Since $ A_1 \in I \subseteq H $ and $ C \in K $, we have one of the following cases for some $ U \in K $
\[
\textup{ (i) }    \    A_0 = B  U   \;  \wedge  \;  U A_1 = C 
\;   ,     \      \      \   
\textup{ (ii) }    \ B = A_0 U   \;   \wedge   \;   U C = A_1  
\          .
\]
In case of (i) we have $  B \preceq_{ K }  A_0 $ which implies $ B \in H $ since $ A_0 \in I $. 
In case of (ii) we have $  U  \preceq_{ K }  A_1 $ which implies $ U \in H $ since $ A_1 \in I $.  
Since $H$ is closed under matrix multiplication and $ A_0 \in I \subseteq H$, we have  $ B = A_0 U  \in H $.
Hence, $ A_0 A_1 \in I $. 
Thus, $ I $ is closed under matrix multiplication.

We show that $ I $ is downward closed under $ \preceq_{ K } $. 
So, assume $ B   \preceq_{ K } A $ where $ A \in I $. 
We need to show that $ B \in I $.
That is, we need to show that  $ B \in H $ and  $ \forall  D  \preceq_{ K }   B \;   [   \    D  \in H  \   ]   $.
Since $ A \in I $ and   $ B   \preceq_{ K } A $, it follows from the definition of $I$ that $ B \in H $. 
Assume now $ B = DC $ where $ C \in K $.
We need to show that $ D \in H $. 
Since $ B   \preceq_{ K } A $, there exists $E \in K $ such that $ A = B E $. 
Hence, $ DC E = BE  = A $. 
Since $ C, E \in K $ and $K$ is closed under matrix multiplication, $ CE \in K $. 
Hence, $ D   \preceq_{ K } A $.
Then, $ D \in H $ since $ A \in  I $. 
Thus, $ I$ is  downward closed under   $ \preceq_{ K } $. 
\end{proof}

\section{Commutative Semirings I}
\label{CommutativeSemiringPropertiesII}

We complete   our proof of interpretability of  $ \overline{  \mathsf{ID} }  $    in    $   \mathsf{IQ}   $ 
by  showing that     $ \mathsf{IQ} $ and $ \mathsf{IQ}^{ (2) } $  are mutually interpretable.

\begin{theorem}  \label{CommutativeSemiring}

 $ \mathsf{IQ} $ and  $ \mathsf{IQ}^{ (2) } $  are mutually interpretable. 

\end{theorem}

\begin{proof}

Since $  \mathsf{IQ}^{ (2) } $ is an extension of  $ \mathsf{IQ} $,  
we only  need to show that   $  \mathsf{IQ}^{ (2) } $ is interpretable  in $ \mathsf{IQ} $.
Our strategy is to first restrict the universe of $ \mathsf{IQ} $ to  an inductive class $N_2 $ which is such that each  of  the axioms  (I)-(VIII) in Figure  \ref{StringsAsMatricesFigure}  holds on $N_2 $ when we restrict quantification  to $N_2$ and  treat addition and multiplication as partial functions.
Recall that a  class $X$ is inductive if $ 0 \in X $ and $ \forall x \in X  \;   [    \   \mathrm{S} x \in X    \    ]   $.
Now, since the axioms of  $  \mathsf{IQ}^{ (2) } $ are all universal sentences, to interpret  $  \mathsf{IQ}^{ (2) } $ in $ \mathsf{IQ} $,
it suffices  to relativize quantification to a subclass $N$ of $N_2$ that is closed under   $  0, \mathrm{S} , +, \times  $.

We start by  restricting the universe of $ \mathsf{IQ} $ to a subclass $N_0 $  where 
$0$ is the only element with an additive inverse, 
and addition is associative and right-cancellative. 
Let  $  u  \in N_0 $ if and only if 
\begin{itemize}

\item[(1)]       $    0  + u = u      $

\item[(2)]       $   \forall x  \;   [    \   x+u = 0   \rightarrow   (   \   x=0  \;  \wedge  \;  u = 0    \     )       \     ]      $

\item[(3)]       $ \forall x   \;   [    \    \mathrm{S}  x  +  u = \mathrm{S}   ( x + u )    \     ]  $

\item[(4)]  $  \forall x y  \;   [    \    (x+y ) + u = x + (y+u )    \    ]   $

\item[(5)]  $  \forall x y    \;   [    \  x + u = y + u \rightarrow  x= y    \    ]        \;  $

\item[(6)]  $      0 u = 0        \;  $.

\end{itemize}

We verify  that $ 0 \in N_0 $. 
We need to show that $0$ satisfies (1)-(6). 
By  $ \mathsf{Q}_4 \equiv   \   \forall  x \;   [    \  x+0  = x   \   ]   $,  
$0$ satisfies (1)-(5). 
By $ \mathsf{Q}_6 \equiv   \   \forall  x \;   [    \  x0  = 0   \   ]   $, 
$0$ satisfies (6).
Thus, $ 0 \in N_0 $.

We verify that $ N_0 $ is closed under   $   \mathrm{S}  $.
Let $ u \in N_0 $. 
We need to show that $  \mathrm{S}   u \in  N_0 $. 
That is, we need to show that $ \mathrm{S}   u $ satisfies (1)-(6).
We have 
\[
0 + \mathrm{S}   u = \mathrm{S}   (0 + u ) = \mathrm{S}   u  = \mathrm{S}   u + 0 
\]
 where the first   equality holds  by  
 $ \mathsf{Q}_5 \equiv   \   \forall  x  y  \;   [    \  x+ \mathrm{S}   y  = \mathrm{S}   (x+y )   \   ]   $, 
 the second  equality holds since $u$ satisfies (1)  and the last equality holds by $   \mathsf{Q}_4 $. 
Thus,   $  \mathrm{S}  u $ satisfies (1).

By $ \mathsf{Q}_5   $  and  $ \mathsf{Q}_2   \equiv   \   \forall  x    \;   [    \     \mathrm{S}    x  \neq  0   \   ]   $
\[
 x +  \mathrm{S}  u =  \mathrm{S}  (x+u ) \neq 0 
 \           .
 \]
Thus,    $   \mathrm{S}  u $ satisfies (2).

We have 
\[
 \mathrm{S}   x +    \mathrm{S}    u = \mathrm{S}  (   \mathrm{S}   x + u ) 
 =   \mathrm{S}    \mathrm{S}   (x+u)  
=  \mathrm{S}    ( x +   \mathrm{S}   u ) 
 \]
 where the first equality holds by  $ \mathsf{Q}_5  $, 
 the second equality  holds since  $u$ satisfies (3) 
 and the last equality  holds by  $ \mathsf{Q}_5  $. 
 Thus,    $ \mathrm{S}  u $  satisfies (3).

We have 
\[
(x+y ) +     \mathrm{S}    u =     \mathrm{S}     ( (x+y) + u )  
= 
  \mathrm{S}     ( x + (y+u) ) 
= 
x + ( y +      \mathrm{S}     u ) 
\]
where the first equality holds by  $ \mathsf{Q}_5  $, 
 the second equality holds since  $u$ satisfies (4), 
and the last equality  holds by  $ \mathsf{Q}_5  $.  
Thus,  $\mathrm{S}  u $  satisfies (4).

We have 
\[
  \mathrm{S}     ( x+u )  = x +     \mathrm{S}    u = y +     \mathrm{S}    u =      \mathrm{S}   ( y+u ) 
\  \Rightarrow    \  
x+u = y+u 
\   \Rightarrow    \  
x + y  
\]
where the first implication follows from 
$ \mathsf{Q}_1   \equiv   \   \forall xy \;   [    \    \mathrm{S}    x =    \mathrm{S}   y   \rightarrow  x = y     \   ]  $, 
and the last implication follows from the assumption that $u $ satisfies (5). 
Thus, $ \mathrm{S}  u $ satisfies (5).

By   $ \mathsf{Q}_7  \equiv   \ \forall x y \;  [    \  x   \times    \mathrm{S}   y =   x  \times  y  + x       \        ]  $
\[
0   \times    \mathrm{S}   u = 0u + 0 = 0 + 0 = 0
\]
where the second  equality follows from the assumption that $u$  satisfies (6) and the last equality holds by $ \mathsf{Q}_4 $.
Thus,  $ \mathrm{S}  u $ satisfies (6).

Since $   Su $ satisfies (1)-(6),     $  \mathrm{S}  u \in N_0 $. 
Thus, $ N_0 $ is closed under  $  \mathrm{S} $.
Since $N_0 $ contains $0$ and is closed under $  \mathrm{S} $,    the class $N_0 $ is inductive.

We restrict $N_0$  to a subclass $N_1 $ where   addition is commutative, 
the left distributive law holds, 
and there are no zero divisors. 
Let  $  u  \in N_1 $ if and only if   $ u \in N_0 $ and 
\begin{itemize}

\item[(7)]       $ \forall x \in N_0      \;   [    \  x + u  = u + x    \     ]  $

\item[(8)]       $ \forall x  \in N_0   \;   \forall    y      \;   [    \   x (y+u ) = xy + xu    \     ]  $

\item[(9)]       $ \forall x   \in N_0   \;   [    \    x     u  =  0 \rightarrow   (    \    x = 0 \;   \vee   \;  u = 0    \    )     \     ]  $

\item[(10)]       $ \forall x  \in N_0     \;   [    \    \mathrm{S}  x   \times     u  = xu + u    \     ]       \;   $.

\end{itemize}

Verify that $  N_1 $ contains $0$. 
We need to show that $ 0 \in N_0 $ and that $0$ satisfies (7)-(10). 
Since $N_0 $ is inductive, $ 0 \in N_0 $. 
By $  \mathsf{Q}_4 $ and (1),  $0$ satisfies (7). 
By $ \mathsf{Q}_4 $ and  $ \mathsf{Q}_6 $, $0$ satisfies  (8).
It is obvious that  $0$  satisfies (9).
By $ \mathsf{Q}_6 $ and  $ \mathsf{Q}_4 $, $0$ satisfies  (10).
Since $ 0  $ is an element of $  N_0  $ and    satisfies (7)-(10),    $ 0 \in N_1 $.

We verify that $N_1 $ is closed under  $  \mathrm{S}   $.
Let $ u \in N_1 $. 
We need to show that $ \mathrm{S}    u \in N_0 $ and that $  \mathrm{S}   u $ satisfies (7)-(10). 
Since $ N_0 $ is inductive and $  u \in N_1 \subseteq N_0 $,  $  \mathrm{S}    u  \in N_0 $. 
We verify that $ \mathrm{S}   u $ satisfies (7). 
Let $ x \in N_0 $. 
Then 
\[
x + \mathrm{S}   u =  \mathrm{S}   ( x+u ) = \mathrm{S}    ( u + x )    = \mathrm{S}   u + x
\]
where the first equality holds by $ \mathsf{Q}_5 $, 
the second equality holds since  $u$ satisfies (7), 
and the last equality holds since $x $ satisfies (3). 
Thus,  $ \mathrm{S}   u $ satisfies (7).

We verify that $ \mathrm{S}  u $ satisfies (8). 
Let $ x  \in N_0 $.
We have 
\[
\begin{array}{r c l c c c l}
x ( y + \mathrm{S}  u )  
& = &     x   \times     \mathrm{S}  (y+ u  )
&  &   &   &        (      \mathsf{Q}_5      )  
\\
&= &   x (y+u )  +  x  
&  &   &   &        (      \mathsf{Q}_7      )  
\\
&= &   ( xy+  xu )  +  x  
&  &   &   &        (   u  \mbox{ satisfies }  (8)     \;          )  
\\
&= &   xy+  ( xu  +  x   ) 
&  &   &   &        (   x  \mbox{ satisfies }  (4)     \;          )  
\\
&= &   xy+  (   x \times   \mathrm{S}    u    )  
&  &   &   &        (      \mathsf{Q}_7      )       \          .
\end{array}
\]
Thus, $ \mathrm{S} u $ satisfies (8).

We verify that $ \mathrm{S}   u $ satisfies (9).
Let $ x \in N_0 $ and assume $ x \times  \mathrm{S}  u = 0 $. 
By $ \mathsf{Q}_7 $,  $  x u  + x = 0 $. 
Since $ x $ satisfies  (2),   $ x = 0 $. 
Thus,   $ \mathrm{S} u $ satisfies (9).

Finally, we  verify that $ \mathrm{S}   u $ satisfies (10).
Let $ x \in N_0 $. 
We have 
\[
\begin{array}{r c l c c c l}
\mathrm{S}  x \times  \mathrm{S}  u 
& = &     ( \mathrm{S}  x   \times     u  )  +    \mathrm{S}  x 
&  &   &   &        (      \mathsf{Q}_7      )  
\\
&= &   (  xu + u )  +   \mathrm{S}  x   
&  &   &   &        (   u  \mbox{ satisfies }  (10)     \;          )  
\\
&= &   xu + (u + \mathrm{S}  x )  
&  &   &   &          (  \mathrm{S}  x \in N_0   \mbox{ satisfies }  (4)     \;          )   
\\
&= &  xu +   \mathrm{S}  (u + x )     
&  &   &   &        (  \mathsf{Q}_5      )   
\\
&= &  xu +  \mathrm{S}    (  x    +   u )     
&  &   &   &        (   x \in N_0   \mbox{ and  }    u    \mbox{ satisfies }  (7)     \;          )   
\\
&= &   xu + (  x +  \mathrm{S}  u  )     
&  &   &   &       (  \mathsf{Q}_5      )    
\\
&= &   (xu +  x )  + \mathrm{S}  u     
&  &   &   &        (   \mathrm{S}  u \in N_0     \mbox{ satisfies }  (4)     \;          )   
\\
&= &   ( x \times  \mathrm{S}  u    )   +   \mathrm{S}  u     
&  &   &   &        (    \mathsf{Q}_7        )            \       .
\end{array}
\]
Thus, $ \mathrm{S}  u $ satisfies (10).

Since $ \mathrm{S}   u $ is an element of  $ N_0 $ and  satisfies (7)-(10),     $ \mathrm{S} u \in N_1 $. 
Thus, $ N_1 $ is closed under  $ \mathrm{S}  $. 
Since $N_1 $ contains $0$ and is closed under $\mathrm{S} $, the class $N_1 $ is inductive.

We restrict $N_1$ to a subclass $N_2 $ where  multiplication is associative and  commutative. 
Let  $  u  \in N_2 $ if and only if   $ u \in N_1 $ and 
\begin{itemize}

\item[(11)]       $ \forall x ,  y  \in N_1    \;   [    \   (xy )  u =  x (yu)      \     ]  $

\item[(12)]       $ \forall x  \in N_1      \;   [    \  x u = u x     \    ]     \;   $.

\end{itemize}

We verify that $0 \in N_2 $. 
We need to show that $ 0 \in N_1 $ and that $0$  satisfies (11)-(12). 
Since $ N_1 $ is inductive, $ 0 \in N_1 $. 
By $ \mathsf{Q}_6 $, $ 0 $ satisfies (11). 
By $  \mathsf{Q}_6 $  and (6),   $0$ satisfies (12). 
Thus, $ 0 \in N_2 $.

We verify that $N_2 $ is closed under  $  \mathrm{S}   $.
Let $ u \in N_2 $. 
We need to show that $  \mathrm{S}    u \in N_1 $ and that $  \mathrm{S}    u $ satisfies (11)-(12). 
Since $ N_{ 2} \subseteq N_1 $ and $N_1 $ is inductive, $  \mathrm{S}    u \in N_1 $. 
We verify that $  \mathrm{S}    u $ satisfies (11). 
Let $ x, y \in N_1 $. 
We have 
\[
\begin{array}{r c l c c c l}
(xy)  \times   \mathrm{S}    u 
& = &   (xy)u + xy  
&  &   &   &        (      \mathsf{Q}_7      )  
\\
&= &    x (yu)  + xy    
&  &   &   &        (   u  \mbox{ satisfies }  (11)     \;          )  
\\
&= &   x ( yu + y )  
&  &   &   &          (   x \in N_0  \mbox{ and } y \in N_1    \mbox{ satisfies }  (8)     \;          )   
\\
&= &  x (y  \times   \mathrm{S}    u ) 
&  &   &   &        (  \mathsf{Q}_7      )                \       .
\end{array}
\]
Thus, $  \mathrm{S}   u $ satisfies (11).

We verify that $  \mathrm{S}    u $ satisfies (12). 
Let $ x \in N_1 $. 
We have 
\[
\begin{array}{r c l c c c l}
x \times   \mathrm{S}    u 
& = &   xu + x 
&  &   &   &        (      \mathsf{Q}_7      )  
\\
&= &    ux + x    
&  &   &   &        (   u  \mbox{ satisfies }  (12)     \;          )  
\\
&= &    \mathrm{S}   u \times  x  
&  &   &   &          (   u \in N_0  \mbox{ and }   x \in N_1    \mbox{ satisfies }  (10)     \;          )             \       .
\end{array}
\]
Thus, $  \mathrm{S}   u $ satisfies (12).

Since $  \mathrm{S}    u $ is an element of  $ N_1 $ and  satisfies (11)-(12),     $  \mathrm{S}   u \in N_2 $. 
Thus, $ N_2 $ is closed under  $   \mathrm{S}    $.
Since $N_2 $ contains $0$ and is closed under $   \mathrm{S}   $, the class $N_2 $ is inductive.

We are almost done. 
All that remains is to restrict $N_2 $ to an inductive  class that is closed under addition and multiplication. 
We start by ensuring closure under addition. 
Let 
\[
N_3 = \lbrace u \in N_2  :   \   
\forall x \in N_2     \;   [    \   x+u  \in N_2     \         ]        
\     \rbrace  
\        .
\]
By $ \mathsf{Q}_4 $, $ 0 \in N_3 $. 
We show that  that $N_3 $ is closed under $ \mathrm{S} $. 
Let $ u \in N_3 $. 
Since $ N_2   $ is inductive and $ u \in N_3 \subseteq N_2 $,  $ \mathrm{S}   u \in N_2 $. 
By $ \mathsf{Q}_5 $,  given $ x \in N_2 $, 
we have  $ x + \mathrm{S}  u = \mathrm{S}   (x + u )  $. 
Since $ u \in N_3 $,  $ x+u \in N_2 $. 
Since $ N_2 $ is inductive, $  \mathrm{S}   (x + u )  \in N_2 $. 
Hence, $ \mathrm{S}  u \in N_3 $. 
Thus, $ N_3 $ is closed under  $ \mathrm{S}  $.

We verify that $ N_3 $ is closed under $ + $. 
Let $ u, v \in N_3 $. 
We need to show that $ u+v \in N_2 $ and  $ \forall x \in N_2  [    \  x + (u+v )   \in N_2  \    ]   $. 
Since $ u \in N_2 $ and $ v \in N_3 $, $ u+v \in N_2 $. 
Now, let $ x \in N_2 $. 
Since $ u \in N_3 $, $ x + u \in N_2 $. 
Since $ v \in N_3 $,  $  (x+u) + v \in N_2 $. 
Since $ v \in N_2 \subseteq N_0 $ satisfies (4),  $   (x+u) + v = x + (u+v) $. 
Hence, $ u+ v \in N_3 $
Thus, $ N_3 $ is closed under $+$.

Let 
\[
N =  \lbrace u    \in N_3 :    \   
\forall x  \in N_3  \;    [     \   xu  \in N_3     \    ]  
\     \rbrace 
\        .
\]
We show that   $N$ is an inductive class that is closed under   $ + $ and $ \times $.
We show that $ 0 \in N $. 
Since $ N_3$ is inductive, $ 0 \in N_3 $.
Let $ x \in N_3 $. 
By $ \mathsf{Q}_6 $,  $  x   0 = 0 \in N_3 $. 
Thus, $ 0 \in N $.

We show that $N$ is closed under   $  \mathrm{S}   $.
Let $ u \in N $. 
We need to show that $ \mathrm{S}   u \in N $. 
Since $ u \in N_3 $ and $N_3 $ is inductive,  $  \mathrm{S}   u \in N_3 $. 
Let $ x \in N_3 $. 
By $ \mathsf{Q}_7 $,   $ x \times  \mathrm{S}   u = xu + x $.
Since $u \in N $, $ xu \in N $. 
Since $ N_3 $ is closed  under addition, $ xu + x \in N_3 $. 
Hence, $ \mathrm{S}   u \in  N $. 
Thus,  $N$ is closed under   $ \mathrm{S}    $.

We show that $N$ is closed under  $  +  $ . 
Let $ u , v \in N \subseteq N_3 $. 
Since $ N_3 $ is closed under addition, $ u+v \in N_3 $. 
Let $ x \in N_3 $. 
Since $ u, v  \in N $, $ xu , xv  \in N_3 $. 
Since $ N_3 $ is closed under addition, $ xu + xv \in N_3 $. 
Since $ x \in N_3 \subseteq N_0 $ and $ v \in N_3 \subseteq N_1 $, 
$ xu + xv = x (u+v ) $. 
Hence, $ u+v \in N $. 
Thus, $ N$ is closed under $  +  $.

We show that $N$ is closed under  $  \times $. 
Let $ u , v \in N \subseteq N_3 $. 
Since $u \in  N_3 $ and $ v \in N $,  $ uv \in N_3 $. 
Let $ x \in N_3 $. 
Since $ u \in N $, $ xu \in N_3 $. 
Since $ v \in N $, $ (xu ) v \in N_3 $. 
Since  $x, u \in N_3 \subseteq N_1 $ and  $ v  \in N_3 \subseteq N_2 $ satisfies  (11), $ (xu) v = x (uv) $.
Hence, $ uv \in N $. 
Thus, $ N$ is closed under $  \times  $.

Since $N$ satisfies the domain conditions and all the axioms of $ \mathsf{IQ}^{ (2 ) } $ 
hold restricted to $N$ as  they are universal sentences, 
 $ \mathsf{IQ}^{ (2 ) } $  is interpretable in  $ \mathsf{IQ} $. 
 Since  $ \mathsf{IQ}^{ (2 ) } $ is an extension of  $ \mathsf{IQ} $, 
 it follows that  $ \mathsf{IQ} $  and  $ \mathsf{IQ}^{ (2 ) } $  are mutually interpretable.
\end{proof}

\section{Commutative Semirings II}
\label{CommutativeSemiringPropertiesIII}

It is clear that  $ \mathsf{Q}^{ (2) } $ is interpretable in  $ \mathsf{Q} $ since each axiom of  $ \mathsf{Q}^{ (2) } $
is provable in $ I \Delta_0 $,  which is $ \mathsf{Q} $ extended with an induction schema for $ \Sigma_0$-formulas, 
and $ I \Delta_0 $ is interpretable in $ \mathsf{Q} $ (see Section V.5c of Hajek  \& Pudlak   \cite{HajekandPudlak2017}).
Lemma V.5.11 of  \cite{HajekandPudlak2017} shows that we can interpret  any finite subtheory of  $ I \Delta_0 $ in  $ \mathsf{Q} $
by restricting the universe of $ \mathsf{Q} $ to a suitable subclass. 
It then follows that   our interpretation of $ \mathsf{TC}^{ \varepsilon } $ in  $ \mathsf{Q}^{ (2) } $ really extends to a recursion-free interpretation of   $ \mathsf{TC}^{ \varepsilon } $ in  $ \mathsf{Q} $.
For the benefit of the reader,
we show that we can also prove this by building on the proof of Theorem  \ref{CommutativeSemiring}.

Given a sentence  $ \phi $ and a class $M  $, 
let $ \phi^M $ denote the sentence we obtain by restricting quantification to $M$.

\begin{theorem}

There exists a class $M$ such that $ \mathsf{Q}  \vdash  \phi^M$ for each axiom $ \phi $ of    $  \mathsf{Q}^{ (2)  }  $.

\end{theorem}

\begin{proof}

Let $ N $ be the class in the proof of Theorem  \ref{CommutativeSemiring}. 
Let 
\[
u     \leq_{ N  }    v   \equiv    \    \exists r  \in N   \;   [     \    u + r  = v     \    ] 
\          .
\]
We restrict  $N $ to an inductive subclass $M_0  $ that is downward closed under   $   \leq_{ N  }   $.
Let 
\[
M_0 =  \lbrace  u    \in N   :    \    \forall  v  \leq_{ N }   u   \;   [       \   v  \in N    \    ]     
\     \wedge     \  
\forall  x, y    \leq_{ N }    u   \;   [    \      x  \leq_{ N }   y   \;   \vee   \;   y   \leq_{ N }    x    \        ]
    \         \rbrace
\          .
\]
We show that $ 0 \in M_0 $. 
Assume $   v + r   = 0 $. 
If $ r = 0 $, then $ v = 0 $ by $ \mathsf{Q}_4 $.
If $ v  \neq 0 $, then by   $ \mathsf{Q}_3 $   there exists $  t $ such   $  r  =  \mathrm{S} t  $. 
Then, by $ \mathsf{Q}_5 $,  $  0 = v +r =  \mathrm{S}  ( v + t  )  $    which contradicts $  \mathsf{Q}_2 $. 
Thus, since $ 0 \in N $ and $ 0+ 0  = 0 $, we have $ 0 \in M_0 $.

We show that $ M_0 $ is closed under $  \mathrm{S} $. 
Let $ u \in M_0 $. 
We need to show that $  \mathrm{S} u \in M_0 $. 
Since $ u \in M_0 \subseteq N $ and $N$ is inductive,  $ \mathrm{S} u \in N $. 
We  show that  $ \forall  v  \leq_{ N }   \mathrm{S} u   \;   [       \   v  \in N    \    ]     $.
Assume  $ r \in N $ and $  v  + r  =    \mathrm{S}   u $. 
We need to show that  $ v \in N $. 
If $ v = 0 $, then $ v \in N $ since $N$ is an inductive class. 
Otherwise, by $ \mathsf{Q}_3 $, there exists $ w $ such that $ \mathrm{S} w = v $. 
By Clause (3) in   the proof of Theorem  \ref{CommutativeSemiring}  
\[
 \mathrm{S}   u   = v+r  =     \mathrm{S} w  + r   =     \mathrm{S}   ( w + r ) 
 \        .
 \]
By  $ \mathsf{Q}_1 $,  $  w  + r  = u $. 
Hence, $ w \leq_{ N }   u $. 
 Since $ u \in M_0 $,   we have    $ w \in N $. 
 Since $ N $ is an inductive class,   $   v =    \mathrm{S} w  \in N $. 
 Thus,   $ \forall  v  \leq_{ N }   \mathrm{S} u   \;   [       \   v  \in N    \    ]     $.

We show that    
$ \forall  x, y    \leq_{ N }  \mathrm{S}    u   \;   [    \      x  \leq_{ N }   y   \;   \vee   \;   y   \leq_{ N }    x    \        ]  $.
Assume $   x, y    \leq_{ N }  \mathrm{S}    u  $. 
By what we  have just shown,   $ x, y \in N $. 
If $ x =  \mathrm{S}    u  $ or $ y = \mathrm{S}    u  $, then $ x $ and $y$ are comparable with respect to  $ \leq_{ N } $
since   $   x, y    \leq_{ N }  \mathrm{S}    u  $. 
Otherwise, by $ \mathsf{Q}_4 $   
\[
  \mathrm{S}    u = x + r    \    \    \wedge    \      \     \mathrm{S}    u  = y + t 
\      \         \mbox{  where    }     r, t \in N \setminus \lbrace 0 \rbrace 
  \        .
  \]
Since $x, y, r, t \in N $, we have 
\[
  \mathrm{S}    u = x + r  = r + x     \    \    \wedge    \      \        \mathrm{S}    u  = y + t  = t + y 
  \]
by Clause (7) in the proof of   Theorem  \ref{CommutativeSemiring}. 
By $  \mathsf{Q}_3 $, there exist $ r_0, t_0 $ such that $ r =  \mathrm{S}  r_0 $ and $ t = \mathrm{S}  t_0 $.
Hence
\[
  \mathrm{S}    u =  \mathrm{S} r_0  +  x   =     \mathrm{S} ( r_0  +  x   )    
     \    \    \wedge    \      \   
        \mathrm{S}   u  = \mathrm{S}  t_0 + y      =    \mathrm{S}   ( t_0 + y  )  
  \]
by Clause (3)  in the proof of   Theorem  \ref{CommutativeSemiring}. 
By $  \mathsf{Q}_1 $,  
 $     u =  r_0  +  x      $   and   $   u  =    t_0 + y    $.
 Hence, $ r_0, t_0 \leq_{ N } u $ which implies $ r_0 , t_0 \in N $ since $ u \in M_0 $. 
 Then
 \[
     u =  r_0  +  x   =  x + r_0         \       \    \wedge    \           \          u  =    t_0 + y  = y + t_0    
     \]
     by   Clause (7) in the proof of   Theorem  \ref{CommutativeSemiring}. 
   Hence, $ x, y \leq_{ N } u $ which implies that $ x $ and $y$ are comparable with respect to $ \leq_{ N } $ since $ u \in M_0 $. 
Thus, we have    $ \forall  x, y    \leq_{ N }  \mathrm{S}    u   \;   [    \      x  \leq_{ N }   y   \;   \vee   \;   y   \leq_{ N }    x    \        ]  $.
It then follows that $  \mathrm{S} u \in M_0 $.

Since $ \leq_{ N }  $ is transitive, $ M_0 $ is downward closed under   $ \leq_{ N } $. 
Indeed, assume   $   w  \leq_{ N }  v $ and    $  v  \leq_{ N } u $.
Then,   there exist $ r, t \in N $ such that $ v =  w + r  $ and $ u =  v + t $. 
Hence, $ u = (w+r ) + t $. 
Since $ t \in N \subseteq N_0 $,  we have 
\[
   u = (w+r ) + t = w + (r + t ) 
   \]
by Clause (4)  in   the proof of Theorem  \ref{CommutativeSemiring}.
Since $ r, t \in N $ and $N$ is closed under addition,  $ r + t \in N $. 
Hence, $ w  \leq_{ N } u $. 
Thus,  $ \leq_{ N }  $ is transitive.

We restrict $M_0$ to a subclass $M_1$ that is closed under addition. 
Let 
\[
M_1  =  \lbrace    u  \in M_0  :    \       \forall x \in  M_0       \;   [    \   x+u  \in M_0     \         ]          \           \rbrace
\        .
\]
The class $M_1 $ is shown to be closed under $0, \mathrm{S} $ and $ +$ just as in the proof of Theorem  \ref{CommutativeSemiring}.
We show that $ M_1 $ is downward closed under $ \leq_{ N } $. 
Assume $ u \in M_1 $ and $ u = v + r $ where $ r \in N $.  We need to show that $ v \in M_1 $. 
So, let  $ x \in M_0 $. 
We need to show that $ x + v \in M_0 $. 
We have 
\[
 M_0   \ni    x + u =  x + (v+r )   =  (x+v ) + r 
 \]
  by Clause (4)  in the proof of Theorem  \ref{CommutativeSemiring}.
Then 
\[
x+v \leq_{ N }  x+u    \in M_0 
\      .
\]
Since $ M_0 $ is downward closed under $ \leq_{ N } $, we have $ x+v \in M_0 $. 
Hence, $ v \in M_1 $. 
Thus,    $ M_1 $ is downward closed under $ \leq_{ N } $.

Finally, we restrict $M_1$ to  a  domain $M$. 
Let 
\[
M  =  \lbrace    u  \in M_1  :    \       \forall x \in  M_1       \;   [    \   xu  \in M_1     \         ]          \           \rbrace
\        .
\]
The class $M $ is shown to be closed under $0, \mathrm{S} ,  +  ,  \times $ just as in the proof of Theorem  \ref{CommutativeSemiring}.
We show that $ M_1 $ is downward closed under $ \leq_{ N } $. 
Assume $ u \in M  $ and $ u = v + r $ where $ r \in N $.  We need to show that $ v \in M $. 
So, let  $ x \in M_1 $. 
We need to show that $ x v \in M_1 $. 
We have 
\[
 M_1 \ni    x u =  x (v+r )   = xv  + xr  
 \]
  by Clause (8)  in the proof of Theorem  \ref{CommutativeSemiring}.
Since $ v \leq_{ N } u $,  $ u \in M \subseteq M_1 $ and $ M_1 $ is downward closed under $   \leq_{ N }  $, 
we have $   v  \in M_1 $.
Then, by Clause  (7)     in the proof of Theorem  \ref{CommutativeSemiring}
\[
 u = v +r = r + v 
 \        .
 \]
Hence, $ r \leq_{ N } u $ which implies $ r \in M_1 $. 
Since $ x, r \in M_1 $ and $ M_1 $ is closed under $ \times $, we have $ xr \in M_1 \subseteq N $. 
Then,  $  x u =  xv  + xr    $ implies $ xv   \leq_{ N }  xu  $. 
Since $ xu \in M_1 $ and $M_1 $ is downward closed under $ \leq_{ N } $,  we have $ xv \in M_1 $. 
Hence, $ v \in M $. 
Thus, $ M$ is    downward closed under $ \leq_{ N } $.

Axioms (I)-(VIII) in Figure    \ref{StringsAsMatricesFigure}  and the axioms of $ \mathsf{Q} $ that are universal sentences 
hold on $M$ when we restrict quantification to $M$ 
since they hold on  $N$  when we restrict quantification to  $N$. 
We show that  $  \mathsf{Q}_3  \equiv   \    \forall x  \;   [   \  x = 0  \;   \vee  \;   \exists  y   \;   [     \    x   =  \mathrm{S} y   \   ]   $
holds on $M$. 
Assume $ x \in M \setminus \lbrace 0 \rbrace  $. 
By   $  \mathsf{Q}_3 $, there exists $ y $ such that $ x =  \mathrm{S} y $. 
We need to show that $ y \in M $. 
By $ \mathsf{Q}_4 $ and $  \mathsf{Q}_5 $,  we have 
\[
  x =  \mathrm{S} y  =    \mathrm{S} y  + 0   =  \mathrm{S} ( y  + 0 )   =  y +  \mathrm{S} 0 
  \      .
  \]
Since $ M$ is an inductive class, $  \mathrm{S} 0   \in M  \subseteq    N $. 
Hence,  $  y  \leq_{   N  }   x  $. 
Since $M$ is downward closed under $  \leq_{   N  }    $, we have $ y  \in M $. 
Thus,  $  \mathsf{Q}_3    $ holds restricted to $M$.

Finally, we show the trichotomy law
$ \forall x y  \;   [    \     x  <_{ \mathsf{l} }  y   \;   \vee   \;      x = y       \;   \vee   \;   y  <_{ \mathsf{l} }  x    \      ]    $
holds restricted to  $M$.
Recall that  
$   x <_{ \mathsf{l} }   y  \equiv   \   \exists  r  \;   [     \   r \neq 0  \;  \wedge   \;    r + x = y      \       ]    $.
Let $ x, y \in M $. 
Since $M$ is closed under addition and addition on $M$ is commutative
\[
 y +x =  x+ y \in M 
 \         .
 \] 
Then, $ x, y   \leq_{ N }    x+ y    $. 
Since $M \subseteq M_0 $,  we have 
\[
  x  \leq_{ N }  y    \;   \vee   \;    y    \leq_{ N }    x  
  \       .
  \]
Assume $ y = x + r  $ where $ r \in N $. 
By Clause (7) in the proof of Theorem    \ref{CommutativeSemiring},
$ y = x+r = r +x $. Hence, $ r \leq_{ N } y $ which implies $ r \in M $. 
Similarly, if $ x = y +t $ where $ t \in N$, then $ t \in M $. 
Hence, since  $  x  \leq_{ N }  y    \;   \vee   \;    y    \leq_{ N }    x  $  holds  
\[
\exists  r, t \in M    \;      [      \       y = x + r        \   \vee   \     x = y + t       \        ]  
\        .
\]
Thus,  the trichotomy law   holds restricted to  $M$. 
\end{proof}

\end{document}